\documentclass[a4paper,reqno,12pt]{amsart}
\usepackage{mathrsfs}
\usepackage{cases}
\usepackage{epic}
\usepackage{amsfonts}
\usepackage{graphicx}
\usepackage{amsmath}
\usepackage{amssymb, upgreek}
\usepackage{bm}
\usepackage{latexsym,todonotes}
\usepackage{pdflscape}
\usepackage[all]{xypic}
\usepackage{ytableau}
\usepackage[enableskew,vcentermath]{youngtab}
\usepackage[all]{xy}
\usepackage{color}
\usepackage{colordvi}
\usepackage{multicol}
\usepackage[linktocpage=true]{hyperref}
\hypersetup{colorlinks,linkcolor=blue,urlcolor=cyan,citecolor=blue}

\allowdisplaybreaks

\begin{document}
	\input xy
	\xyoption{all}

	\newtheorem{innercustomthm}{{\bf Theorem}}
	\newenvironment{customthm}[1]
	{\renewcommand\theinnercustomthm{#1}\innercustomthm}
	{\endinnercustomthm}
	
	\newtheorem{innercustomcor}{{\bf Corollary}}
	\newenvironment{customcor}[1]
	{\renewcommand\theinnercustomcor{#1}\innercustomcor}
	{\endinnercustomthm}
	
	\newtheorem{innercustomprop}{{\bf Proposition}}
	\newenvironment{customprop}[1]
	{\renewcommand\theinnercustomprop{#1}\innercustomprop}
	{\endinnercustomthm}

	\def \la{\lambda}
	\newcommand{\LaK}{\la_{\texttt{Kr}}}
	\newcommand{\Q}{\mathbb Q}
	\newcommand{\QK}{Q_{\texttt{Kr}}}
	\newcommand{\rad}{\operatorname{rad}\nolimits}
	\newcommand{\cone}{\operatorname{Cone}\nolimits}
	
	\newcommand{\LaJ}{\La_{\texttt J}}
	\newcommand{\tMHL}{\cs\cd\widetilde{\ch}(\LaK^\imath)}
	\newcommand{\tMHg}{{}^\imath\widetilde{\ch}(\QJ)}
	\newcommand{\PL}{\bbP^1_{\bfk}}
	\def\bbP{\mathbb P}
	\def\bfk{\Bbbk}
	\renewcommand{\mod}{\operatorname{mod^{\rm nil}}\nolimits}
	\newcommand{\Aut}{\operatorname{Aut}\nolimits}
	
	\newcommand{\End}{\operatorname{End}\nolimits}
	\newcommand{\Iso}{\operatorname{Iso}\nolimits}
	
	\newcommand{\sqq}{{\bf v}}
	\newcommand{\bq}{{q}}
	
	\newcommand{\tMHLJ}{{}^\imath\widetilde{\ch}(\bfk \QJ)}
	
	\newcommand{\bB}{{\bf B}}
	\newcommand{\iH}{{H}^{\imath}}
	\newcommand{\iP}{P^{\imath}}
	\newcommand{\iB}{{B}^{\imath}}
	\newcommand{\iQ}{{Q}^{\imath}}
	\newcommand{\iV}{{V}^{\imath}}
	\newcommand{\iVh}{\widehat{V}}
	\newcommand{\is}{{s}^{\imath}}
	\newcommand{\haE}{\widehat{E}}
	\newcommand{\haH}{\widehat{H}}
	\newcommand{\haP}{\widehat{P}}
	\newcommand{\haQ}{\widehat{Q}}
	\newcommand{\haV}{\widehat{V}}
	\newcommand{\haT}{\widehat{\Theta}}
	\newcommand{\ka}{\kappa}
	\newcommand{\vv}{v}
	\newcommand{\x}{{\bf x}}
	
	\newcommand{\ov}{\overline}
	\newcommand{\und}{\underline}
	\newcommand{\tk}{\widetilde{k}}
	\newcommand{\tK}{\widetilde{K}}
	\newcommand{\tTT}{\operatorname{\widetilde{\texttt{\rm T}}}\nolimits}
	\newcommand{\iRH}{\operatorname{{}^\imath \widetilde{\ch}}\nolimits}

	\newcommand{\aut}{\operatorname{Aut}\nolimits}
	\newcommand{\res}{\operatorname{res}\nolimits}
	
	\allowdisplaybreaks
	
	\newcommand{\QJ}{Q_{\texttt J}}
	\def\bfk{\Bbbk}
	\def\calc {\mathcal C}
	
	\def \cI{\mathcal I}
	\def \cJ{\mathcal J}
	\def \cR{\mathcal R}
	\def \th{\widetilde{\ch}}
	
	\newcommand{\rep}{\operatorname{rep}\nolimits}
	\newcommand{\Ext}{\operatorname{Ext}\nolimits}
	\newcommand{\Hom}{\operatorname{Hom}\nolimits}
	\renewcommand{\Im}{\operatorname{Im}\nolimits}
	\newcommand{\Ker}{\operatorname{Ker}\nolimits}
	\newcommand{\Coh}{\operatorname{Coh}\nolimits}
	\newcommand{\Id}{\operatorname{Id}\nolimits}
	
	\newcommand{\coker}{\operatorname{Coker}\nolimits}
	\newcommand{\qbinom}[2]{\begin{bmatrix} #1\\#2 \end{bmatrix} }
	
	\newcommand{\gr}{\operatorname{gr}\nolimits}
	\newcommand{\id}{\operatorname{Id}\nolimits}
	\newcommand{\Res}{\operatorname{Res}\nolimits}
	\def \tT{\widetilde{\mathcal T}}
	\def \tTL{\tT(\la^\imath)}
	\def \och{\check{\ch}}
	
	\newcommand{\mbf}{\mathbf}
	\newcommand{\mbb}{\mathbb}
	\newcommand{\mrm}{\mathrm}
	
	\newcommand{\LR}[2]{\left\llbracket \begin{matrix} #1\\#2 \end{matrix} \right\rrbracket}
	\newcommand{\ext}{{ \mathfrak{Ext}}}
	\def\scrP{\mathscr{P}}
	\newcommand{\bk}{{\mathbb K}}
	\newcommand{\cc}{{\mathcal C}}
	\newcommand{\gc}{{\mathcal GC}}
	\newcommand{\dg}{{\rm dg}}
	\newcommand{\ce}{{\mathcal E}}
	\newcommand{\cs}{{\mathcal S}}
	\newcommand{\cP}{{\mathcal P}}
	\newcommand{\cl}{{\mathcal L}}
	\newcommand{\cf}{{\mathcal F}}
	\newcommand{\cx}{{\mathcal X}}
	\newcommand{\cy}{{\mathcal Y}}
	\newcommand{\ct}{{\mathcal T}}
	\newcommand{\cu}{{\mathcal U}}
	\newcommand{\cv}{{\mathcal V}}
	\newcommand{\cn}{{\mathcal N}}
	\newcommand{\mcr}{{\mathcal R}}
	\newcommand{\ch}{{\mathcal H}}
	\newcommand{\ca}{{\mathcal A}}
	\newcommand{\cb}{{\mathcal B}}
	\newcommand{\ci}{{\I}_{\btau}}
	\newcommand{\cj}{{\mathcal J}}
	\newcommand{\cm}{{\mathcal M}}
	\newcommand{\cp}{{\mathcal P}}
	\newcommand{\cg}{{\mathcal G}}
	\newcommand{\cw}{{\mathcal W}}
	\newcommand{\co}{{\mathcal O}}
	\newcommand{\cq}{{Q^{\rm dbl}}}
	\newcommand{\cd}{{\mathcal D}}
	\newcommand{\ck}{\widetilde{\mathcal K}}
	\newcommand{\calr}{{\mathcal R}}
	\newcommand{\dmno}{\ytableausetup{boxsize=3pt}\ydiagram{1,1}}
	\newcommand{\dmnoB}{\ytableausetup{boxsize=3pt}\ydiagram{2}}
	\newcommand{\mno}{\ytableausetup{boxsize=3pt}\ydiagram{1}}
	\newcommand{\iLa}{\la^{\imath}}
	\newcommand{\La}{\Lambda}
	\newcommand{\ol}{\overline}
	\newcommand{\ul}{\underline}
	\newcommand{\ow}{\widetilde}
	
	\newtheorem{theorem}{Theorem}[section]
	\newtheorem{acknowledgement}[theorem]{Acknowledgement}
	\newtheorem{algorithm}[theorem]{Algorithm}
	\newtheorem{assumption}[theorem]{Assumption}
	\newtheorem{axiom}[theorem]{Axiom}
	\newtheorem{case}[theorem]{Case}
	\newtheorem{claim}[theorem]{Claim}
	\newtheorem{conclusion}[theorem]{Conclusion}
	\newtheorem{condition}[theorem]{Condition}
	\newtheorem{conjecture}[theorem]{Conjecture}
	\newtheorem{construction}[theorem]{Construction}
	\newtheorem{corollary}[theorem]{Corollary}
	\newtheorem{criterion}[theorem]{Criterion}
	\newtheorem{definition}[theorem]{Definition}
	\newtheorem{example}[theorem]{Example}
	\newtheorem{exercise}[theorem]{Exercise}
	\newtheorem{lemma}[theorem]{Lemma}
	\newtheorem{notation}[theorem]{Notation}
	\newtheorem{problem}[theorem]{Problem}
	\newtheorem{proposition}[theorem]{Proposition}
	\newtheorem{solution}[theorem]{Solution}
	\newtheorem{summary}[theorem]{Summary}
	\numberwithin{equation}{section}
	
	\theoremstyle{remark}
	\newtheorem{remark}[theorem]{Remark}
	
	\def \cz{\mathcal Z}
	
	\def \bmu{\boldsymbol{\mu}}
	\def \bnu{\boldsymbol{\nu}}
	\def \bla{\boldsymbol{\la}}
	
	\def \bfK{{\mathbf K}}
	
	\def \bA{{\mathbf A}}
	\def \ba{{\mathbf a}}
	\def \bL{{\mathbf L}}
	\def \bF{{\mathbf F}}
	\def \bS{{\mathbf S}}
	\def \bC{{\mathbf C}}
	\def \bU{{\mathbf U}}
	\def \bc{{\mathbf c}}
	\def \fpi{\mathfrak{P}^\imath}
	\def \Ni{N^\imath}
	\def \fp{\mathfrak{P}}
	\def \fg{\mathfrak{g}}
	\def \fk{\fg^\theta}  
	
	\def \fn{\mathfrak{n}}
	\def \fh{\mathfrak{h}}
	\def \fu{\mathfrak{u}}
	\def \fv{\mathfrak{v}}
	\def \fa{\mathfrak{a}}
	\def \fq{\mathfrak{q}}
	\def \Z{{\mathbb Z}}
	\def \F{{\mathbb F}}
	\def \D{{\mathbb D}}
	\def \bB{{\mathbb B}}
	\def \C{{\mathbb C}}
	\def \N{{\mathbb N}}
	\def \Q{{\mathbb Q}}
	\def \G{{\Bbb G}}
	\def \P{{\Bbb P}}
	\def \K{{\mathbf k}}
	\def \bK{{\mathbb K}}
	
	\def \E{{\Bbb E}}
	\def \A{{\Bbb A}}
	\def \L{{\Bbb L}}
	\def \R{{\Bbb R}}
	\def \I{{\Bbb I}}
	\def \BH{{\Bbb H}}
	\def \T{{\Bbb T}}
	\def \de{{\delta}}
	\def \vth{{\theta}}
	
	\def \cN{{\mathcal N}}
	
	\newcommand{\arxiv}[1]{\href{http://arxiv.org/abs/#1}{\tt arXiv:\nolinkurl{#1}}}

	\newcommand{\nc}{\newcommand}
	\newcommand{\browntext}[1]{\textcolor{brown}{#1}}
	\newcommand{\greentext}[1]{\textcolor{green}{#1}}
	\newcommand{\redtext}[1]{\textcolor{red}{#1}}
	\newcommand{\bluetext}[1]{\textcolor{blue}{#1}}
	\newcommand{\brown}[1]{\browntext{ #1}}
	\newcommand{\green}[1]{\greentext{ #1}}
	\newcommand{\red}[1]{\redtext{ #1}}
	\newcommand{\blue}[1]{\bluetext{ #1}}
	
	\newcommand{\wtodo}{\todo[inline,color=orange!20, caption={}]}
	\newcommand{\lutodo}{\todo[inline,color=green!20, caption={}]}

	\title[Double Hall-Littlewood symmetric polynomials]{Double Hall-Littlewood symmetric functions}

	\author[Jiayi Chen]{Jiayi Chen}
	\address{Department of Mathematics, Shantou University, Shantou 515063, P.R. China}
	\email{chenjiayi@stu.edu.cn}

	\author[Ming Lu]{Ming Lu}
	\address{Department of Mathematics, Sichuan University, Chengdu 610064, P.R. China}
	\email{luming@scu.edu.cn}

	\author[Shiquan Ruan]{Shiquan Ruan}
	\address{School of Mathematical Sciences, Xiamen University, Xiamen 361005, P.R. China}
	\email{sqruan@xmu.edu.cn}

	\subjclass[2020]{Primary  18E35, 16G20, 05E05}
	\keywords{Hall and derived Hall algebras, Pieri rules, double Hall-Littlewood functions, Schur Laurent functions}

	\maketitle

\begin{quote}\begin{center}
{\em Dedicated to Professor Bangming Deng on the occasion of his 60th birthday}
\end{center}
\end{quote}
	\begin{abstract}
		We establish a ring isomorphism between the derived Hall algebra of the Jordan quiver and the ring of double symmetric functions (i.e., the ring of symmetric polynomials in two sets of countably many variables, invariant under the respective actions of their symmetric groups) with a parameter $t$. This isomorphism maps the derived Hall basis (the natural basis of the derived Hall algebra) to a class of double Hall-Littlewood (HL) symmetric functions, which are formulated via raising and lowering operators. These double HL functions are parameterized by bipartitions; they reduce to the classical HL functions when one of the partitions is empty, and specialize to Schur Laurent symmetric functions at $t = 0$. We also derive the Pieri rules for these double HL functions. Additionally, we obtain several natural generating functions for the derived Hall algebra as well as their transition relations, which can be transferred to the ring of double symmetric functions via the established ring isomorphism.
	\end{abstract}

	\setcounter{tocdepth}{1}
	\tableofcontents
	
	\section{Introduction}

	\subsection{Hall algebras and Hall-Littlewood symmetric functions}

	The classical Hall algebras, known as the Hall algebra of the Jordan quiver, can be dated back to the early 20th century with Steinitz's research on commutative $p$-groups \cite{St01}; see also \cite{Ha57}. This algebra is isomorphic to the ring of symmetric polynomials in infinitely many variables (i.e., the ring of symmetric functions), and its Hall basis corresponds to Hall-Littlewood (HL) symmetric functions (cf. \cite{Mac95}). 
	
	Ringel \cite{Rin90} developed Hall algebras for any abelian categories, thereby providing a categorical realization of the positive parts of quantum groups of Dynkin type; see \cite{Gr95}  for the  extension to Kac-Moody type. 
	In order to realize the whole quantum groups via Hall algebras, several variants of Hall algebras are introduced for triangulated categories; see e.g. \cite{Ka97,X97,PX97,T06,XX08,XC13}. In particular, derived Hall algebras were formulated in \cite{T06,XX08} for any triangulated category satisfying the (left) homological-finiteness condition. The derived Hall algebra for odd-periodic triangulated categories was also given by Xu-Chen \cite{XC13}. 
	
	Recently, the authors \cite{CLR25} established derived Hall algebras for root categories; compare with \cite{Zh25}. For a finitary hereditary abelian category $\ca$, the derived Hall algebra of its root category was constructed by counting the triangles and using the octahedral axiom, which was proved to be isomorphic to the Drinfeld double of the Hall algebra of $\ca$. 
	
	%
	%
	\subsection{Derived Hall algebra of the Jordan quiver}
	
	
	Let $\QJ$ be the Jordan quiver, that is, the quiver with a single vertex $1$ and a single loop $\alpha:1\rightarrow 1$. Let $\ca={\mathrm {rep}}^\mathrm{nil}_\bfk(\QJ)$ be the category of finite-dimensional nilpotent  representations over the finite field $\bfk=\F_q$. The isomorphism classes of $\ca$ can be parametrized by all partitions. Denote by $S^{(\lambda)}$ the module for any partition $\lambda$. Let $\th(\bfk\QJ)$ be the Hall algebra of $\ca$, and $\Lambda_t=\Q(t)[x_1,x_2,\dots]^{\mathfrak{S}_\infty}$ be  the ring of symmetric polynomials. Then we have a $\Q(q^{-1})$-algebra isomorphism $\th(\bfk\QJ)\rightarrow \Lambda_{q^{-1}}$, sending $[S^{(\lambda)}]$ to the HL symmetric function $V_\lambda$ for any partition $\lambda$; see \cite{Mac95}. Here $\Lambda_{q^{-1}}$ is the specialization of $\Lambda_t$ at $t=q^{-1}$.  
	
	Let $\mcr$ be the root category of $\ca$, and $[1]$ the shift functor; see \cite{PX97}. Then the isomorphism classes of $\mcr$ can be parametrized by bipartitions $(\lambda,\mu)$, and the corresponding objects are $S^{(\lambda)}\oplus S^{(\mu)}[1]$. Let $\cd\widetilde{\ch}(\bfk\QJ)$ be the derived Hall algebra of $\mcr$; see \cite{CLR25}. Let $\cd\Lambda_t= \Q(t) [x_1^+, x_2^+, \ldots,x_1^-,x_2^-,\ldots]^{\mathfrak{S}_\infty\times \mathfrak{S}_\infty}$ be the ring of double symmetric polynomials. 
	Then there exists an algebra isomorphism $\cd\widetilde{\ch}(\bfk\QJ)\rightarrow \cd\Lambda_{q^{-1}}$ by specializing $\cd\Lambda_{t}$ at $ t=q^{-1}$
    ; see \cite[Corollary 5.5]{CLR25}.

	The ring $\cd\Lambda_t$ of double symmetric functions has been extensively investigated in the literature; see e.g. \cite{BR87,Ch05,BLM15}. 
	The $\cd\Lambda_t$ can be embedded into the ring of diagonal symmetric functions $\Q(t) [x_1^+, x_2^+, \ldots,x_1^-,x_2^-,\ldots]^{\mathfrak{S}_\infty}$, i.e., the ring of symmetric polynomials in two sets of countably many variables under the simultaneous symmetric group action. The category of nilpotent representations of the Jordan quiver can be viewed as fundamental building blocks of the category of coherent sheaves on an elliptic curve. Both diagonal symmetric functions and coherent sheaves on an elliptic curve have been shown to be closely related to the double affine Hecke algebra \cite{Ch05, SchV11}.

    \subsection{The goal}
	The goal of this paper is to introduce double HL functions $V_{\lambda,\mu}$, which correspond to the natural basis $[S^{(\lambda)}\oplus S^{(\mu)}[1]]$ of $\cd\widetilde{\ch}(\QJ)$.
	Inspired by \cite{Mac95,Tam11,LRW25}, the double HL functions are defined by 
	using raising and lowering operators, called the Giambelli type formulas. These double HL functions are parameterized by bipartitions; they reduce to the classical HL functions when one of the partitions is empty. We also derive the Pieri rules for these double HL functions. 
	
	It is remarkable that another class of double HL functions has been constructed as a limit of double Macdonald symmetric functions in \cite{BLM15}, which is completely different from ours.

	\subsection{Double Hall-Littlewood symmetric functions}
	The raising operator formalism played a fundamental role in the Schubert calculus of classical type developed by Buch, Kresch and Tamvakis \cite{BKT08}. Tamvakis further developed the raising operator approach \cite{Tam11}; in particular, starting from a raising operator definition of Giambelli polynomials (and then HL functions), he found a direct proof of the horizontal Pieri rule \cite[III, (5.7$'$)]{Mac95}. Such a Pieri rule was earlier obtained by Morris \cite{Mor64}. The lowering operator appeared in \cite{Lec06}. The $\imath$HL symmetric functions, derived from the $\imath$Hall algebra of the Jordan quiver, are introduced in \cite{LRW25} by the synergistic actions of raising and lowering operators, and specialized to universal characters of type C.
	
	With the help of Tamvakis' approach, we are able to formulate an analogue for the ring $\Q(t)[v_1^\pm,v_2^\pm,\ldots]$, which can be identified with the ring $\cd\Lambda_t$ of double symmetric functions (by interpreting $v_r^\pm$ as the classic symmetric functions $q_r$ or $h_r$ in positive/negative part). Unlike the original definition of these operators, which applied only to a single partition \cite{Lec06, LRW25}, the raising and lowering operators are revised to act on bipartitions $(\lambda,\mu)$, thereby yielding the double HL symmetric functions $V_{\lambda,\mu}$.

	Similar to Tamvakis’ work, we prove that all double HL symmetric functions form a basis for the ring of double symmetric functions and satisfy two mirror identities. Consequently, we derive two Pieri rules for $\cd\Lambda_t$, corresponding to the formulas governing the multiplication of the generators of the positive/negative parts with the double HL symmetric functions. 
	Moreover, when specializing at $t=0$, double HL symmetric functions coincide with the Schur-Laurent functions given in \cite{SV15}.
	
	In the derived Hall algebra $\cd\widetilde{\ch}(\bfk \QJ)$, we also have Pieri rules, which arise as special cases of the derived Hall multiplication formula. Then the algebra isomorphism $\cd\widetilde{\ch}(\bfk\QJ)\rightarrow \cd\Lambda_{q^{-1}}$ sends the natural basis $[S^{(\lambda)}\oplus S^{(\mu)}[1]]$ to the double HL functions $V_{\lambda,\mu}$. We also introduce generating functions of several natural generating
	sets in $\cd\widetilde{\ch}(\bfk\QJ)$. 
	Based on this isomorphism, we extend the classic generating functions (such as elementary/complete symmetric functions) to $\cd\Lambda_t$, and consider their transition relations. 
	
In the following, let us describe the main results more precisely.


	%
	%
	\subsection{Main results}
	
	Consider the ring {$\mathcal{D}\Lambda_t=\Q(t)[v_1^\pm, v_2^\pm, \ldots]$}  of double symmetric functions. For two partitions $\rho,\nu$, we construct the double HL symmetric function $V_{\rho,\nu}$ via raising and lowering operators $R_{ij}^\pm$ and $L_{ij}$ acting on $v_{\rho,\nu}=v^+_{\rho}v^-_{\nu}$, where $v_{\rho}^+=v^+_{\rho_1}v^+_{\rho_2}\cdots$ and $v^-_{\rho}= v^-_{\nu_1}v^-_{\nu_2}\cdots$. 
	
	For partitions $\mu, \nu$, we denote $\mu \leq \nu$ if $\mu_i \le \nu_i$ for all $i$, and denote $\mu \stackrel{a}{\rightarrow} \nu$ if $\mu \leq \nu$ and $\nu-\mu$ is a horizontal $a$-strip. The definitions for the polynomials $\varphi_{\nu/\mu}(t), \psi_{\nu/\mu}(t)$ can be found in \eqref{eq:phipsi}, following \cite[III, (5.8)-(5.8$'$)]{Mac95}. Then the horizontal Pieri rules are formulated for double HL symmetric functions.
	
	\begin{customthm}{{\bf A}}
		[Horizontal Pieri rules, Theorem \ref{thm:Pieri-Ho}] 
		\label{ThmA}
		For $r\ge 1$ and any partitions $\rho,\nu$, we have
		\begin{align*}
			V_{\rho,\nu}  \cdot \vv^+_r
			&=\sum_{a+b=r} \sum_{\mu \stackrel{a}{\rightarrow} \nu ,\rho\stackrel{b}{\rightarrow} \la} \, \varphi_{\nu/\mu}(t) \psi_{\la/\rho}(t) \;  V_{\la,\mu},
	\\
			V_{\rho,\nu}  \cdot \vv^-_r
			&=\sum_{a+b=r} \sum_{\nu \stackrel{a}{\rightarrow} \mu ,\la\stackrel{b}{\rightarrow} \rho} \, {\psi_{\mu/\nu}(t) \varphi_{\rho/\la}(t)} \;  V_{\la,\mu}.
		\end{align*}
	\end{customthm}
	
	For partitions $\mu, \nu$, we use $\nu \stackrel{b}{\downarrow} \mu$ to denote that $\mu \leq \nu$ and $\nu-\mu$ is a vertical $b$-strip. The definitions for $\varphi_r(t)$, $b_{\mu}(t)$, and $f_{\mu,(1^{b})}^\nu (t)$ can be found in \eqref{eq:phim}--\eqref{def:bla} and \eqref{eq:fmu}, following \cite[III, (2.12), (3.2)]{Mac95}. The vertical Pieri rules are also formulated.
	
	\begin{customthm}{{\bf B}}
		[Vertical Pieri rules, cf. Theorem  \ref{prop:haPieri-v}]
		For $r\geq1$ and any partitions $\rho,\nu$, we have{
			\begin{align*}
				V_{\rho,\nu}\cdot V^+_{(1^r)}
				=&\sum_{a+b=r} \sum_{\la\stackrel{a}{\downarrow} \rho} \sum_{\nu \stackrel{b}{\downarrow}\mu}
				\ \frac{b_{\rho}(t)}{b_\la(t)}\varphi_r(t)  f_{\rho,(1^{a})}^\la(t) f_{\mu,(1^b)}^\nu(t) \cdot V_{\la,\mu},
		\\
				V_{\rho,\nu}\cdot V^-_{(1^r)}
				=&\sum_{a+b=r} \sum_{\rho\stackrel{a}{\downarrow} \la} \sum_{\mu \stackrel{b}{\downarrow}\nu}
				\  \frac{b_{\nu}(t)}{b_\mu(t)}\varphi_r(t)  f_{\la,(1^{a})}^\rho(t) f_{\nu,(1^b)}^\mu(t) \cdot V_{\la,\mu}.
		\end{align*}}
	\end{customthm}
	
	Denote by $\cd\th(\bfk \QJ)$ the {derived Hall algebra} of the root category of the Jordan quiver and $\cd\th( \QJ)$ its generic version; see \S\ref{subsec:g-derived}. The algebra $\cd\th(\QJ)$ is generated by infinitely many generators
	$ \fu_{{(r),\emptyset}}, \fu_{\emptyset,{(r)}}$ for $r\ge 1$, and admits  a $\Q(t)$-basis $\{\widehat{V}_{\la,\mu}\mid \lambda,\mu\text{ are partitions}\}$. Here $\widehat{V}_{\lambda,\mu}$ is a slightly rescaling of the Hall basis element $\fu_{\lambda,\mu}$; see \eqref{eq:haV}. 
	
	\begin{customthm}{{\bf C}} [Theorem \ref{thm:iso}, Corollary \ref{cor:isomorphism}]\label{ThmC}
		There exists a $\Q(t)$-algebra isomorphism $\Phi_t: \cd\th(\QJ) \longrightarrow \mathcal{D}\La_t$ such that
		\begin{align*}
			\Phi_t (\fu_{{(r),\emptyset}}) =t^{-r} v^+_r ,\qquad\Phi_t (\fu_{\emptyset,{(r)}}) = v^-_r \quad \text{  for  }r\ge 1.
		\end{align*}
		Moreover, for any partitions $\la,\mu$, we have{
			\begin{align*}
				\Phi_t(\widehat{V}_{\la,\mu})= V_{\la,\mu}.
		\end{align*}}
	\end{customthm}

The (horizontal) Pieri rule in Theorem \ref{ThmA} determines the double HL symmetric functions $V_{\lambda,\mu}$. In order to prove Theorem \ref{ThmC}, we only need to establish the Pieri rule for $\widehat{V}_{\lambda,\mu}$ in $\cd\th(\QJ)$, which is given in Proposition \ref{prop:PieriHall-Ho}.
	
	
	To formulate some generating functions below, we need to consider the algebras over the field $\Q(\sqq)$, where $\sqq =\sqrt{q}$. Let $y,z$ be two indeterminates. Recall $\varphi_m(t)$ and $\exp_q \big(\dfrac{x}{1-q} \big)$ from \eqref{eq:phim} and \eqref{eq:Euler}. In \eqref{gen1}--\eqref{def:p}, we introduce the generating functions $\widetilde{H}(y,z), \widetilde{E}(y,z)$ and $\widetilde{\Theta}(y,z)$ of several natural generating sets in the derived Hall algebra $\cd\th(\bfk\QJ)$ (compare with \cite{BKa01, Sch06, LRW23, LRW25}). Following \cite{Mac95,BKa01,Sch06}, we denote by $H_i(y),E_i(y),\Theta_{i}(y)$ for $i=1,2$ the classic generating functions in the positive/negative parts of $\cd\th(\bfk\QJ)$; see \eqref{eq:classic-SFTheta}--\eqref{eq:classic-SFH}.

	
	\begin{customthm}{{\bf D}}
		[Propositions \ref{e-h}, \ref{theta-e-h}, Theorem \ref{prop:transition}]\label{ThmD}
		The following identities hold in the derived Hall algebra of the Jordan quiver $\cd\widetilde{\ch}(\bfk \QJ) \otimes_{\Q} \Q(\sqq)$:
		\begin{align*}
			\widetilde{H}(y,z)&\widetilde{E}(-y,-z)=\big(\mathrm{exp}_q(\frac{yz}{1-q})\big)^2,\\
			\widetilde{\Theta}(y,z)&=(1-qyz)^2\frac{\widetilde{E}(-y,-z)}{\widetilde{E}(-qy,-qz)}={(1-yz)^{-2}}\frac{\widetilde{H}(qy,qz)}{\widetilde{H}(y,z)},\\
				\widetilde{\Theta}(y,z)&=\Theta_1(y)\Theta_2(z)\exp\big( \sum_{k\ge1}\frac{1-q^k}{k}(yz)^k\big),
		\\
		\widetilde{E}(y,z)&=E_1(y)E_2(z)\exp_q\big( \frac{yz}{1-q}\big),
		\\
		\widetilde{H}(y,z)&=H_1(y)H_2(z)\exp_q\big( \frac{yz}{1-q}\big).
		\end{align*}
	\end{customthm}

	Thanks to the isomorphism in Theorem \ref{ThmC}, we can transfer the generating functions and their transition relations established in Theorem \ref{ThmD} to the ring $\cd\Lambda_t$ of double symmetric functions.
	
	%
	%


	\subsection{Organization}
	In Section~\ref{sec:Pieri}, we establish the horizontal Pieri rule for the Giambelli type polynomials $V_{\la,\mu}$, generalizing the approach of Tamvakis. We then make suitable identifications of $v_r$ to formulate the double HL symmetric functions. In Section~\ref{sec:Jordan}, we review derived Hall algebra $\cd\widetilde{\ch}(\bfk \QJ)$ of the Jordan quiver and then establish horizontal/vertical Pieri rules in $\cd\widetilde{\ch}(\bfk \QJ)$.
	In Section~\ref{sec:generating}, we study several distinguished generating sets for the derived Hall algebra $\cd\widetilde{\ch}(\bfk \QJ)$ in the form of generating functions and establish their transition relations. 

	\vspace{2mm}
	
	\noindent{\bf Acknowledgment.}
	JC is partially supported by the National Natural Science Foundation of China Tianyuan Fund for Mathematics (No. 12526562).
    ML is partially supported by the National Natural Science Foundation of China (No. 12171333).
	SR is partially supported by
	Fundamental Research Funds for Central Universities of China (No. 20720250059), Fujian Provincial Natural Science Foundation of China (No. 2024J010006) and
	the National Natural Science Foundation of China (Nos. 12271448 and 12471035).

	\section{Double Hall-Littlewood symmetric functions}
	\label{sec:Pieri}
	
	In this section,  we introduce the Giambelli type polynomials and establish their Pieri rules for double HL symmetric functions, and compare the latter with classical HL functions.
	
	Denote by 
	$\N,\Z,\Q$ the sets of non-negative integers, integers, and rational numbers, respectively.

	\subsection{A Giambelli type formula}
	Let $\vv_r^\pm$ be commuting variables of degree $r$, for $r \in \Z_{\ge 1}$. Let us set $\vv_0^\pm=1$ and $\vv_a^\pm=0$, for $a<0$. We define
	\[
	\vv_\alpha^\pm =\vv_{\alpha_1}^\pm\vv_{\alpha_2}^\pm \cdots\quad \text{  and }\quad \vv_{\alpha,\beta}=\vv_\alpha^+\vv_\beta^-
	\]
	multiplicatively for any integer sequences $\alpha =(\alpha_1, \alpha_2, \ldots)$ and $\beta =(\beta_1, \beta_2, \ldots)$ with finitely many nonzero entries.
	Given two integer sequences $\alpha, \beta$, we denote $\alpha \ge \beta$ (or $\beta \le \alpha$) if $\alpha_i \ge \beta_i$, for all $i\ge 1$; in particular, $\alpha \ge 0$ denotes that $\alpha$ is a composition. For a composition $\alpha$, we denote
	\begin{align}
		\label{eq:number}
		\begin{split}
			\ell(\alpha) &= \text{number of nonzero parts in } \alpha,
			\\
			|\alpha| &=\alpha_1 +\alpha_2+\cdots.
		\end{split}
	\end{align}
	
	Let $i,j\ge1$, we define the lowering operator $L_{ij}$ (cf. \cite{Lec06,LRW25}) acting on the pair of integer sequences $(\alpha,\beta)$ by
	\[
	L_{ij} (\alpha,\beta) =((\ldots, \alpha_{i-1}, \alpha_i-1,, \alpha_{i+1}, \ldots),(\ldots, \beta_{j-1}, \beta_j-1,, \beta_{j+1}, \ldots)).
	\]
	And for $1\le i<j$ we have the left raising operator $R^+_{ij}$ acting on $(\alpha,\beta)$ by
	\[
	R_{ij}^+ (\alpha,\beta) =((\ldots, \alpha_i+1, \ldots, \alpha_j-1, \ldots),\beta).
	\]
	Similarly, the right raising operator $R^-_{ij}$ acts on $(\alpha,\beta)$ by
	\[
	R_{ij}^- (\alpha,\beta) =(\alpha,(\ldots, \beta_i+1, \ldots, \beta_j-1, \ldots)).
	\]
	While the raising operators were originally introduced by Young (cf. \cite{Mac95}), some of the significant applications have appeared only recently; see \cite{BKT08, Tam11}. 
	The actions of raising and lowering operators on $\vv_{\alpha,\beta}$ are given by 
	\[
	L_{ij} \vv_{\alpha,\beta} =\vv_{L_{ij}(\alpha,\beta)},\qquad R^\pm_{ij} \vv_{\alpha,\beta} =\vv_{R_{ij}^\pm(\alpha,\beta)}.
	\]
	Note that all raising and lowering operators commute with each other. 
	
	Let $t$ be an indeterminate. Denote by $\cp_n$ the set of partitions $\la$ of $n$ with dominance order $\unrhd$, and we sometimes write $\la \vdash n$. Set $\cp =\cup_{n\ge 0}\cp_n$. 
	The dominance order $\unrhd$ also makes sense on the set of all compositions of $n$. 
	
	The ring
	\[
	\mathcal{D}\La_{t} = \Q(t)[v_1^\pm, v_2^\pm, \ldots] 
	\]
	admits a basis $\{v_{\la,\mu} \mid \la,\mu \in \cp \}$.
	
	Justification of applying raising (and lowering) operators properly can be found in \cite[\S2.2]{Tam12} and \cite[(4.1) and below]{BMPS19}. Let
	\begin{align}
		\label{eq:F}
		F(u)=\frac{1-u}{1-tu}=1+ (1-t^{-1}) \sum_{r\geq1} t^r u^r.
	\end{align}Denote
	\begin{align}  \label{eq:L}
		\L_{ij} &:=   \frac{1 - L_{ij}}{1 -  t L_{ij}}
		=1 + (1-t^{-1}) \sum_{s\ge 1} t^{s} L_{ij}^s ,
		\\
		\qquad
		\R_{ij}^\pm &:=  \frac{1 -R^\pm_{ij}}{1 -tR^\pm_{ij}}
		= 1 + (1-t^{-1}) \sum_{p\ge 1} t^{p} (R_{ij}^\pm)^p,
		\\
		\L &:= \prod_{i,j\ge1} \L_{ij},\qquad
		\R^\pm := \prod_{1\le i<j}\R_{ij}^\pm.
	\end{align}
	
	\begin{definition}
		\label{def:V}
		For any integer vectors $\alpha$ and $\beta$, we define $V_{\alpha,\beta} \in \mathcal{D}\La_{t}$ by:
		\begin{align}  \label{eq:LRq}
			V_{\alpha,\beta} &= \L\cdot\R^+\cdot\R^-(\vv_{\alpha,\beta}).
		\end{align}
		In particular, we denote $V^+_{\alpha}=V_{\alpha,\emptyset}$ and $V^-_\beta=V_{\emptyset,\beta}$, where $\emptyset$ denotes the empty set.
		Note that $V^+_{r} =\vv_{r,\emptyset} =\vv^+_r$ and $V^-_{r} =\vv_{\emptyset,r} =\vv^-_r$. Furthermore, $V^+_{\alpha} =\R^+ (\vv^+_\alpha)$ and $V^-_{\alpha} =\R^- (\vv^-_\alpha)$.
	\end{definition}

	Denote
	\begin{align}   \label{eq:D}
		D_{\ell,\mathfrak{m}} : =
		\prod_{\substack{1\le i\le\ell\\1\le j\le \mathfrak{m}}} \L_{ij} \cdot \prod_{1\le i<j\le\ell} \R_{ij}^+\cdot\prod_{1\le i<j\le\mathfrak{m}} \R_{ij}^- .
	\end{align}
	
	Below we follow \cite[\S1-2]{Tam11} to manipulate the operators $D_{\ell,\mathfrak{m}}$. We have
	\begin{align*}
		D_{\ell,\mathfrak{m}}
		&= D_{\ell-1,\mathfrak{m}} \cdot\prod_{\substack{1\le j\le \mathfrak{m}}} \L_{\ell j} \cdot \prod_{1\le i<\ell} \R_{i\ell}^+
		\\
		&= D_{\ell-1,\mathfrak{m}} \cdot \prod_{1\le j\le \mathfrak{m}} \Big(1 + (1-t^{-1}) \sum_{s\ge 1}  t^s L_{\ell j}^s \Big)\cdot \prod_{1\le i< \ell}
		\Big(1 + (1-t^{-1}) \sum_{p\ge 1} t^{p} (R^+_{i\ell})^p \Big).
	\end{align*}
	Applying $D_{\ell,\mathfrak{m}}$ to $\vv_{(\alpha,r),\beta} $, where $(\alpha,r)$ and $\beta$ are integer vectors with $\ell(\alpha,r)=\ell$ and $\ell(\beta)=\mathfrak{m}$, we obtain
	\begin{align}
		V_{(\alpha,r),\beta} &=D_{\ell,\mathfrak{m}} \vv_{(\alpha,r),\beta}=
		D_{\ell-1,\mathfrak{m}} \cdot\prod_{\substack{1\le j\le \mathfrak{m}}} \L_{\ell j} \cdot \prod_{1\le i<\ell} \R_{i\ell}^+(\vv_{(\alpha,r),\beta})
		\notag \\
		&= D_{\ell-1,\mathfrak{m}}  \cdot \prod_{1\le i<\ell} \R_{i\ell}^+\cdot\prod_{\substack{1\le j\le \mathfrak{m}}} \Big(1 + (1-t^{-1}) \sum_{s\ge 1} t^{s} L_{\ell j}^s \Big)(\vv_{(\alpha,r),\beta})
		\notag\\&= D_{\ell-1,\mathfrak{m}}  \cdot \prod_{1\le i<\ell} \R_{i\ell}^+ (\sum_\gamma t^{|\gamma|}(1-t^{-1})^{\ell(\gamma)} \vv_{(\alpha,r-|\gamma|),\beta-\gamma})
		\notag\\&=\sum_\gamma t^{|\gamma|}(1-t^{-1})^{\ell(\gamma)}\cdot D_{\ell-1,\mathfrak{m}}  \cdot \prod_{1\le i<\ell}\Big(1 + (1-t^{-1}) \sum_{p\ge 1}  t^p (R^+_{i\ell})^p \Big) ( \vv_{(\alpha,r-|\gamma|),\beta-\gamma})
		\notag\\&=\sum_\gamma t^{|\gamma|}(1-t^{-1})^{\ell(\gamma)}\cdot D_{\ell-1,\mathfrak{m}}  (\sum_\eta t^{|\eta|}(1-t^{-1})^{\ell(\eta)}  \vv_{(\alpha+\eta,r-|\gamma|-|\eta|),\beta-\gamma})
		\notag\\&=\sum_{\gamma,\eta} t^{|\gamma|+|\eta|}(1-t^{-1})^{\ell(\gamma)+\ell(\eta)}\cdot D_{\ell-1,\mathfrak{m}}  (\vv_{\alpha+\eta,\beta-\gamma}\vv_{r-|\gamma|-|\eta|}^+)
		\notag\\&=\sum_{\gamma,\eta} t^{|\gamma|+|\eta|}(1-t^{-1})^{\ell(\gamma)+\ell(\eta)}\cdot   V_{\alpha+\eta,\beta-\gamma}\vv_{r-|\gamma|-|\eta|}^+
		,
		\label{eq:recur1}
	\end{align}
	summed over all compositions $ \eta \in \N^{\ell-1}$ and $ \gamma \in \N^{\mathfrak{m}}$.

	Similarly, we also have \begin{align}
		V_{\alpha,(\beta,r)}	=\sum_{\gamma,\eta} t^{|\gamma|+|\eta|}(1-t^{-1})^{\ell(\gamma)+\ell(\eta)}\cdot   V_{\alpha-\gamma,\beta+\eta}\vv_{r-|\gamma|-|\eta|}^-
		\label{eq:recur2}
	\end{align}
	summed over all compositions $ \gamma \in \N^{\ell-1}$ and $ \eta \in \N^{\mathfrak{m}}$. The formulas \eqref{eq:recur1} and \eqref{eq:recur2} give a recursive formula for computing $V_{\alpha,\beta}$. 
	
	
	\begin{proposition}
		$\{V_{\la,\mu} \mid \la,\mu \in \cp \}$ forms a $\Q(t)$-basis for $\mathcal{D}\La_{t}$.
	\end{proposition}
	
	\begin{proof}
		The dominance order $\unrhd$ can be extended to the set $\cp$ of partitions such that $\lambda\unrhd \mu$ if $|\lambda|<|\mu|$; and  furthermore extended to the set $\cp\times \cp$ of bipartitions such that $(\tilde{\lambda},\tilde{\mu})\unrhd (\lambda, \mu)$ if and only if $\tilde{\lambda}\unrhd \lambda$ and $\tilde{\mu}\unrhd \mu$. Then we have $(\tilde{\la},\tilde{\mu})$ strictly dominates $(\la,\mu)$ (denoted by $(\tilde{\la},\tilde{\mu})\triangleright(\lambda,\mu)$) if one of the following cases happens:
		\begin{itemize}
			\item[(a)] $(\tilde{\la},\tilde{\mu})=L_{ij}(\la,\mu)$  for some $i,j$;
			\item[(b)] $(\tilde{\la},\tilde{\mu})=R^+_{ij}(\la,\mu)$  for some  $i<j$;
			\item[(c)] $(\tilde{\la},\tilde{\mu})=R^-_{ij}(\la,\mu)$ for some $i<j$.
		\end{itemize}
		Therefore, by using \eqref{eq:LRq} we obtain that $$V_{\la,\mu}\in \vv_{\la,\mu}+\sum_{(\tilde{\la},\tilde{\mu})\triangleright(\lambda,\mu)}\Q(t)\vv_{\tilde{\la},\tilde{\mu}}$$ 
		where the sum is over bipartitions $(\tilde{\la},\tilde{\mu})$ which strictly dominate $(\lambda,\mu)$. Recall that $\{v_{\la,\mu} \mid \la,\mu \in \cp \}$ forms a $\Q(t)$-basis for $\mathcal{D}\La_{t}$. It follows immediately that $\{V_{\la,\mu} \mid \la,\mu \in \cp \}$ forms a $\Q(t)$-basis for $\mathcal{D}\La_{t}$.
	\end{proof}
	\subsection{Mirror identities}
	
	\begin{proposition}  \label{prop:HT}
		Suppose that an equation (of finite sums) of the form
		$\sum_{\alpha,\beta} a_{\alpha,\beta} V_{\alpha,\beta} =\sum_{\alpha,\beta} b_{\alpha,\beta} V_{\alpha,\beta}$ holds. Then
		$$\sum_{\alpha,\beta} a_{\alpha,\beta} V_{(\mu, \alpha),\beta} =\sum_{\alpha,\beta} b_{\alpha,\beta} V_{(\mu, \alpha),\beta}\quad\text{ and }\quad\sum_{\alpha,\beta} a_{\alpha,\beta} V_{ \alpha,(\mu,\beta)} =\sum_{\alpha,\beta} b_{\alpha,\beta} V_{ \alpha,(\mu,\beta)}$$ 
		for any integer vector $\mu$.
	\end{proposition}
	
	\begin{proof}
		The proof is by induction on $\ell(\mu)$, which is verbatim the same as for \cite[Proposition 1]{Tam11} by using \eqref{eq:recur1} and \eqref{eq:recur2}, hence omitted here.
	\end{proof}

	Given partitions $\mu, \nu$ and $b\in \N$, we use
	\[
	\nu \rightarrow \mu,
	\qquad (\text{respectively, } \nu \stackrel{b}{\rightarrow} \mu)
	\]
	to denote that $\nu \leq \mu$ and $\mu -\nu$ is a horizontal strip (respectively, a horizontal $b$-strip). Denote by $m_i(\mu)$ the number of times $i$ occurs as a part of $\mu$. Denote by $\mu' =(\mu_1', \mu_2', \ldots)$ the conjugate partition of $\mu$.
	
	For a horizontal $r$-strip $\sigma=\la-\nu$,  let $I=I_{\la-\nu}$ (respectively, $J=J_{\la-\nu}$) be the set of integers $i\geq 1$ such that $\sigma_i'=1$ and $\sigma_{i+1}'=0$ (respectively, $\sigma_i'=0$ and $\sigma_{i+1}'=1$). Following \cite[III]{Mac95}, we introduce
	\begin{align}
		\label{eq:phipsi}
		\begin{split}
			\varphi_{\la/\nu}(t) &=\prod\limits_{i\in I_{\la-\nu}}(1-t^{m_i(\la)}),
			\\
			\psi_{\la/\nu}(t) &=\prod\limits_{j\in J_{\la-\nu}}(1-t^{m_j(\nu)}).
		\end{split}
	\end{align}
	
	Our horizontal Pieri rules in Theorem~\ref{thm:Pieri-Ho} below are formulated using the polynomials $\varphi, \psi$ above. To prove that, we shall require  2 mirror identities. The first mirror identity below is similar to Tamvakis \cite[\S2.2]{Tam11}. 
	
	For $\gamma =(\gamma_1, \ldots, \gamma_{\ell+1}) \in \N^{\ell+1}$, we set
	\[
	\gamma_{\bullet} :=(\gamma_1, \ldots, \gamma_\ell).
	\]
	
	\begin{proposition} [Mirror identity I] (also see \cite[\S2.2]{Tam11})
		For an integer $b \ge 0$ and partitions $\rho,\mu$ satisfying $\rho \in \N^\ell$, we have
		\begin{align}
			\sum_{\gamma  \in \N^{\ell+1}, |\gamma| = b} (1-t)^{\ell( \gamma_{\bullet})}  \;  V_{\rho +\gamma,\mu}
			&= \sum_{\rho \stackrel{\scriptscriptstyle b}{\rightarrow} \la} \psi_{\la/\rho}(t) V_{\la,\mu};
			\label{iQ:psi}
			\\
			\sum_{\gamma  \in \N^{\ell+1}, |\gamma| = b} (1-t)^{\ell( \gamma_{\bullet})}  \;  V_{\mu,\rho +\gamma}
			&= \sum_{\rho \stackrel{\scriptscriptstyle b}{\rightarrow} \la} \psi_{\la/\rho}(t) V_{\mu,\la}.
			\label{iQ:psi2}
		\end{align}
	\end{proposition}
	
	\begin{proof}
		Since \eqref{iQ:psi2} is dual to \eqref{iQ:psi}, we only prove the first one here.
		
		For $\mu=\emptyset$, it is 
		\begin{align}
			\sum_{\gamma  \in \N^{\ell+1}, |\gamma| = b} (1-t)^{\ell( \gamma_{\bullet})}  \;  V^+_{\rho +\gamma}
			&= \sum_{\rho \stackrel{\scriptscriptstyle b}{\rightarrow} \la} \psi_{\la/\rho}(t) V^+_{\la},
		\end{align}
		which is the special case of \cite[Proposition 2.8]{LRW25} by setting $\vth=0$ there. In fact, its proof is entirely due to Tamvakis \cite[\S2.2]{Tam11}.
		Then the formula \eqref{iQ:psi} follows from Proposition \ref{prop:HT}.  
	\end{proof}
	
	The second mirror identity below is similar to \cite[Theorem 1]{Tam11}. In our context, this identity is a lowering operator counterpart to the mirror identity \eqref{iQ:psi}; this shall become clear shortly in the proof of Theorem~\ref{thm:Pieri-Ho}.
	
	Recall $\varphi_{\la/\nu}(t)$ from \eqref{eq:phipsi} and $\ell(\alpha)$ from \eqref{eq:number}.
	
	\begin{proposition} [Mirror identity II] (see \cite[Theorem 1]{Tam11})
		For an integer $a \ge 0$ and partitions $\nu,\tilde{\rho}$ satisfying $\nu \in \N^\ell$, we have
		\begin{align}
			\sum_{\beta \in \N^\ell, |\beta|=a} (1-t)^{\ell( \beta)}  \;   V_{\nu -\beta,\tilde{\rho}}
			&= \sum_{\mu \stackrel{a}{\rightarrow} \nu}  \varphi_{\nu/\mu}(t)   V_{\mu,\tilde{\rho}};
			\label{iQ:phi2}
			\\
			\sum_{\beta \in \N^\ell, |\beta|=a} (1-t)^{\ell( \beta)}  \;   V_{{\tilde{\rho},}\nu -\beta}
			&= \sum_{\mu \stackrel{a}{\rightarrow} \nu}  \varphi_{\nu/\mu}(t)   V_{\tilde{\rho},\mu}.
			\label{iQ:phi}
		\end{align}
	\end{proposition}
	
	\begin{proof}
		We only prove the first formula since the other one is dual.
		
		For $\tilde{\rho}=\emptyset$, it is 
		\begin{align}
			\sum_{\beta \in \N^\ell, |\beta|=a} (1-t)^{\ell( \beta)}  \;   V^+_{\nu -\beta}
			&= \sum_{\mu \stackrel{a}{\rightarrow} \nu}  \varphi_{\nu/\mu}(t)   V^+_{\mu},
		\end{align}
		which is given in \cite[Theorem 1]{Tam11}.
		Then the formula \eqref{iQ:phi2} follows from Proposition \ref{prop:HT}.  
	\end{proof}

	\subsection{Horizontal Pieri rules}
	
	
	As we shall see, the following Pieri rules correspond to the multiplication formula \eqref{Pieri:h} in the derived Hall algebra setting.

	\begin{theorem}  [Horizontal Pieri rules] 
		\label{thm:Pieri-Ho}
		For $r\ge 1$ and any partitions $\rho,\nu$, we have
		\begin{align}
			V_{\rho,\nu}  \cdot \vv^+_r
			&=\sum_{a+b=r} \sum_{\mu \stackrel{a}{\rightarrow} \nu ,\rho\stackrel{b}{\rightarrow} \la} \, \varphi_{\nu/\mu}(t) \psi_{\la/\rho}(t) \;  V_{\la,\mu},
			\label{Pieri:hSF3}
		\\
			V_{\rho,\nu}  \cdot \vv^-_r
			&=\sum_{a+b=r} \sum_{\nu \stackrel{a}{\rightarrow} \mu ,\la\stackrel{b}{\rightarrow} \rho} \, {\psi_{\mu/\nu}(t) \varphi_{\rho/\la}(t)} \;  V_{\la,\mu}.
			\label{Pieri:hSF32}
		\end{align}
	\end{theorem}
	
	\begin{proof}
    We only prove \eqref{Pieri:hSF3}. 
		Set $\ell =\ell(\rho)$ and $\mathfrak{m}=\ell(\nu)$. Recall from \eqref{eq:L}--\eqref{eq:D} that 
		$\L_{ij}^{-1} =1 +(1-t) \sum_{k\ge 1} L_{ij}^k$ and $$D_{\ell,\mathfrak{m}}  =
		\prod_{\substack{1\le i\le\ell\\1\le j\le \mathfrak{m}}} \L_{ij} \cdot \prod_{1\le i<j\le\ell} \R_{ij}^+\cdot\prod_{1\le i<j\le\mathfrak{m}} \R_{ij}^-.$$ 
		Then we have
		\begin{align}
			V_{\rho,\nu}  \cdot \vv^+_r
			&= D_{\ell,\mathfrak{m}} (\vv_{(\rho,r),\nu})
			\notag
			\\&=D_{\ell+1,\mathfrak{m}}\cdot\prod_{\substack{1\le j\le \mathfrak{m}}} (\L_{\ell+1,j})^{-1} \cdot \prod_{1\le i\le\ell} (\R_{i,\ell+1}^+)^{-1} (\vv_{(\rho,r),\nu})
			\notag\\&=D_{\ell+1,\mathfrak{m}} \cdot \prod_{1\le i\le\ell} (\R_{i,\ell+1}^+)^{-1}
			\cdot \prod_{\substack{1\le j\le \mathfrak{m}}}(1 +(1-t) \sum_{k\ge 1}  L_{\ell+1,j}^k)(\vv_{(\rho,r),\nu})
			\notag\\&=D_{\ell+1,\mathfrak{m}} \cdot \prod_{1\le i\le\ell} (\R_{i,\ell+1}^+)^{-1}
			\sum_{\beta \in \N^\mathfrak{m}}   (1-t)^{\ell( \beta)}  \;   \vv_{(\rho , r-|\beta|),\nu-\beta}
			\notag\\&=\prod_{1\le i\le\ell} (\R_{i,\ell+1}^+)^{-1}
			\sum_{0\le a\le r}\sum_{\beta \in \N^\mathfrak{m},|\beta|=a}   (1-t)^{\ell( \beta)}  \;   V_{(\rho , r-|\beta|),\nu-\beta}
			\notag\\ \label{iQ:r1} 
			&=\prod_{1\le i\le\ell} (\R_{i,\ell+1}^+)^{-1} \sum_{0\le a\le r} \sum_{\mu \stackrel{a}{\rightarrow} \nu}  \varphi_{\nu/\mu}(t)
			V_{(\rho,r-a), \mu},
		\end{align}
		where we have used the mirror identity \eqref{iQ:phi} with $\tilde{\rho}=(\rho,r-a)$ in the last step.
		
		Observing that $V_{(\rho,r-a),\mu} = D_{\ell+1,\mathfrak{m}}(\vv_{(\rho,r-a),\mu})$, and $D_{\ell+1,\mathfrak{m}}$ commutes with $(\R^+_{i,\ell+1})^{-1}$ for $1\le i\le r$, we can rewrite the RHS of \eqref{iQ:r1} as
		\begin{align}
			\label{iQ:r2}
			V_{\rho,\nu}  \cdot \vv^+_r
			&=\prod_{1\le i\le\ell} (\R_{i,\ell+1}^+)^{-1} \sum_{0\le a\le r} \sum_{\mu \stackrel{a}{\rightarrow} \nu}  \varphi_{\nu/\mu}(t)
			D_{\ell+1,\mathfrak{m}}(\vv_{(\rho,r-a),\mu})
			\notag\\&= \sum_{0\le a\le r}\sum_{\mu \stackrel{a}{\rightarrow} \nu}  \varphi_{\nu/\mu}(t)
			D_{\ell+1,\mathfrak{m}}\cdot\prod_{1\le i \le \ell} (1 +(1-t) \sum_{k\ge 1} (R_{i,\ell+1}^+)^k)(\vv_{(\rho,r-a),\mu})
			\notag\\
			&= \sum_{0\le a\le r}\sum_{\mu \stackrel{a}{\rightarrow} \nu} \varphi_{\nu/\mu}(t)
			D_{\ell+1,\mathfrak{m}}
			(\sum_{\gamma  \in \N^{\ell+1}, |\gamma| =r-a} (1-t)^{\ell( \gamma_{\bullet})} \vv_{(\rho +\gamma),\mu})
			\notag \\
			&= \sum_{0\le a\le r}\sum_{\mu \stackrel{a}{\rightarrow} \nu}  \varphi_{\nu/\mu}(t)
			\cdot\sum_{\gamma  \in \N^{\ell+1}, |\gamma| =r-a} (1-t)^{\ell( \gamma_{\bullet})} V_{(\rho +\gamma),\mu}.
		\end{align}
		The theorem follows now by applying the mirror identity \eqref{iQ:psi} to the RHS of  \eqref{iQ:r2}.
	\end{proof}

	\subsection{Double Hall-Littlewood functions}

	We recall the Hall-Littlewood (HL) functions in variables $x=(x_1, x_2, \ldots)$, following \cite[Pages 208--211]{Mac95}. Denote by $h_r$ the $r$th complete symmetric function, for $r\ge 0$.
	For an indeterminate $u$, denote
	\begin{align*}
		H(u)&= \sum_{r=0}^\infty h_r u^r = \prod_{i\ge 1} \frac1{1-ux_i},
		\qquad
		Q(u) = \sum_{r=0}^\infty Q_{r} u^r=H(u)/H(tu).
	\end{align*}
	
	We set $q_r =Q_r$, and $q_\la =q_{\la_1} q_{\la_2} \cdots$, for any composition $\la$.
	Let $u_1,u_2,\cdots$ be independent indeterminates. Recall $F(u)=(1-u)/(1-tu)$ from \eqref{eq:F}. Then the HL symmetric function $Q_\la$ is the coefficient of $u^\la=u_1^{\la_1} u_2^{\la}\cdots$ in
	\begin{align}
		\label{eq:HL}
		Q(u_1,u_2,\dots)= \prod_{i\geq1}Q(u_i) \prod_{i<j} F(u_iu_j^{-1}).
	\end{align}
	According to \cite[Page 212]{Mac95}, the generating function \eqref{eq:HL} for HL functions $Q_\la$ can be restated that
	\begin{align}  \label{eq:Rq}
		Q_\la = \prod_{i<j} \frac{1 -R_{ij}}{1 -tR_{ij}} q_\la
		=\prod_{i<j} \Big(1 + (t-1) \sum_{r\ge 1} t^{r-1}  R_{ij}^r  \Big) q_\la.
	\end{align}

	%
	We shall set $\vv_r^\pm =q^\pm_r= Q^\pm_{r}$ in the definition \eqref{eq:LRq} of $V_{\la,\mu}$, for any compositions $\la,\mu$.  In this way, we identify the ring $\mathcal{D}\La_{t} = \Q(t)[q_1^\pm, q_2^\pm, \ldots]$ with the ring of double symmetric functions $\Q(t) [x_1^+, x_2^+, \ldots,x_1^-,x_2^-,\ldots]^{\mathfrak{S}_\infty\times \mathfrak{S}_\infty}$. 
	For any partition $\lambda$, we can formulate two versions of HL symmetric functions, denoted by $Q_\lambda^\pm$ in $\cd\Lambda_t$. And we denote by $Q^\pm(u)=\sum_{r=0}^\infty Q^\pm_ru^r$. 
	Then we see from \eqref{eq:LRq} and \eqref{eq:Rq} that if one of the partitions is empty, then $V_{\la,\emptyset}$ or $V_{\emptyset,\la}$ coincides with $Q_\la^+$ or $Q_\la^-$.
	On the other hand, $V_{\la,\mu}$ gives us another class of double symmetric functions, which are called {\em double HL symmetric functions} and denoted by $Q_{\la,\mu}$. Alternatively, in the spirit of the equivalence between \eqref{eq:HL} and \eqref{eq:Rq}, we can define $Q_{\la,\mu}$ as follows.
	
	\begin{definition} [double HL functions]
		\label{def:iQ}
		For any compositions $\la$ and $\mu$, $Q_{\la,\mu}$ is the coefficient of $u^\la w^\mu=u_1^{\la_1}u_2^{\la_2}\cdots w_1^{\mu_1}w_2^{\mu_2}\cdots$ in
		\begin{align}
			\label{HL functions}
			Q(u,w)=&Q(u_1,u_2,\cdots,w_1,w_2,\cdots)
			\\\notag=&\prod_{i\geq1} Q^+(u_i)\prod_{j\geq1} Q^-(w_j)\prod_{i<i'} F(u_i^{-1}u_{i'})\prod_{j<j'} F(w_j^{-1}w_{j'})\prod_{i,j} F(u_iw_j)
			\\\notag
			=&\prod_{i\geq1} Q^+(u_i)\prod_{j\geq1} Q^-(w_j)\prod_{i<i'} \frac{1-u_i^{-1}u_{i'}}{1- tu_i^{-1}u_{i'}}
			\prod_{j<j'} \frac{1-w_j^{-1}w_{j'}}{1- tw_j^{-1}w_{j'}}\prod_{i,j} \frac{1-u_iw_{j}}{1- tu_iw_{j}}.
		\end{align}
	\end{definition}

	\subsection{Modified HL functions}
	Let $x=(x_1, x_2, \ldots)$. Recall the modified Hall-Littlewood functions $H_\la(x; t)$ can be defined using plethysm, $H_\la(x; t) =Q_\la(x/(1-t); t)$, cf. \cite[Ex.~7, Page 234]{Mac95}, where it was denoted by $Q_\la'$. It can also be defined via Garsia's version of Jing's vertex operators \cite{J91}.
	Via raising operators the modified HL functions are defined to be
	\begin{align}  \label{eq:HRh}
		H_\mu = \prod_{1\le i<j}  \frac{1 -R_{ij}}{1 -tR_{ij}} h_\mu.
	\end{align}
	Note $H_{(r)} =h_r$.
	
	In this subsection, we shall take $\vv_r^{\pm} =h_r^{\pm}$ in the definition \eqref{eq:LRq}, and then redenote $V_{\la,\mu}$ as $H_{\la,\mu}$.
	In this way, we identify $\cd\La_{t} = \Q(t)[h_1^{\pm}, h_2^{\pm}, \ldots]$ with the ring of double symmetric functions.
	
	\begin{definition}
		The modified double HL function $H_{\la,\mu} \in \cd\La_{t}$, for any compositions $\la$,$\mu$, is given by
		\begin{align}  
			H_{\la,\mu} &= \L\cdot\R^+\cdot\R^-(h_{\la,\mu})= \L\cdot\R^+\cdot\R^-(h^+_{\la}h^-_{\mu}).
		\end{align}
	\end{definition}
	Then $H_{\la,\emptyset}$ is just the modified HL function $H_\la(x;t)$.

	Now let us set $t=0$, and consider the ring of double symmetric functions $\cd\La := \Q [h_1^\pm, h_2^\pm, \ldots]$. 
	Recall that in \cite{SV15}, for partitions $\la,\mu$, the Schur-Laurent functions $s_{\la,\mu}$ are given by 
	\begin{align}
		\label{schur}
		s_{\la,\mu}=\mathrm{det}\begin{pmatrix}
			h^-_{\mu_s}&h^-_{\mu_s-1}&\cdots &h^-_{\mu_s-s-r+1}\\
			\vdots&\vdots&&\vdots\\
			h^-_{\mu_1+s-1}&h^-_{\mu_1+s-2}&\cdots &h^-_{\mu_1-r}\\
			h^+_{\la_1-s}&h^+_{\la_1-s+1}&\cdots &h^+_{\la_1+r-1}\\
			\vdots&\vdots&&\vdots\\
			h^+_{\la_r-s-r+1}&h^+_{\la_r-s-r+2}&\cdots &h^+_{\la_r}\\
		\end{pmatrix},
	\end{align}
	where $\ell(\la)=r$ and $\ell(\mu)=s$. If $\mu=\emptyset$, then $s_\lambda=s_{\lambda,\emptyset}$ is just the Schur function.  
	On the other hand, we have
	\begin{align}
		\label{t=0}
		H_{\la,\mu}|_{t=0}=\prod_{i,j\ge1} (1-L_{ij})\cdot \prod_{1\le i<j}(1-R_{ij}^+)\cdot \prod_{1\le i<j}(1-R_{ij}^-)(h_{\la,\mu}).
	\end{align}
	
	As a generalization of the well-known result $s_\la=H_\la|_{t=0}$ (see e.g. \cite{Mac95}), we have the following proposition.
	
	\begin{proposition}
		For any partitions $\la$ and $\mu$, we have $
		s_{\la,\mu}=Q_{\la,\mu}|_{t=0}.
		$
	\end{proposition}
	
	\begin{proof}
		Denote by $r=\ell(\lambda),s=\ell(\mu)$. Let $\mathfrak{S}_{r+s}$ be the symmetric group. 
		
		An integer sequence $$\alpha=(\alpha_{-s},\alpha_{-s+1},\cdots,\alpha_{-1},\alpha_1,\alpha_2,\cdots,\alpha_r)$$
		is called a {\em $(r,s)$-composition}   if $\alpha_i\ge0$ for $1\le i\le r$ and $\alpha_j\le0$ for $-s\le j\le -1$; which is furthermore called a {\em $(r,s)$-partition} if $\alpha_{1}\ge\alpha_2\ge\cdots\ge\alpha_r>0$ and $\alpha_{-1}\le\alpha_{-2}\le\cdots\le\alpha_{-s}<0$.
		Given a pair of partitions\[
		\beta=(\beta_1,\cdots,\beta_r) \quad\text{and}\quad \gamma=(\gamma_1,\cdots,\gamma_s),
		\] 
		we can get a $(r,s)$-partition $\la=(-\gamma_s,\cdots,-\gamma_1,\beta_1,\cdots,\beta_r)$. Conversely, any $(r,s)$-partition gives two partitions $\beta$ and $\gamma$ of length $r$ and $s$, respectively.
		
		For any $(r,s)$-composition $\alpha$, define 
		$$h_\alpha:= h^-_{-\alpha_{-s}}\cdots h^-_{-\alpha_{-1}}h^+_{\alpha_1}\cdots h^+_{\alpha_r}.$$
		Therefore,  we can attach a Schur-Laurent symmetric polynomials by \eqref{schur} to the $(r,s)$-composition $\alpha$; or equivalently, we have
		\[
		s_\alpha=\sum_{\omega\in \mathfrak{S}_{r+s}}\epsilon(\omega)h_{\alpha+\delta-\omega(\delta)}
		\]
		where $\epsilon(\omega)\in\{\pm 1\}$ is the sign of the permutation $\omega$, and $\delta$ is the composition  $(r+s-1,r+s-2,\cdots,1,0)$. 
		
		For any $(r,s)$-composition $\alpha$, define 
		\begin{align}
			H_{\alpha}=H_{(\alpha_1,\cdots,\alpha_r),(-\alpha_{-1},\cdots,-\alpha_{-s})}.
		\end{align}
		Define the generalized raising operator $R_{ij}$ to act on $(r,s)$-compositions $\alpha$ by
		\[
		R_{ij} (\alpha) =(\ldots, \alpha_i+1, \ldots, \alpha_j-1, \ldots).
		\]for $i,j\in\{-s,\cdots,-1,1,\cdots,r\}$ and $i<j$. Denote $R_{ij}h_{\alpha}=h_{R_{ij}\alpha}$. Then we can change the formula \eqref{t=0} to 
		\begin{align}
			H_\alpha|_{t=0}=\prod_{i<j}(1-R_{ij})h_\alpha,
		\end{align}
		for any $(r,s)$-partition $\alpha$. It is remarkable that $h_{a}^{\pm}=0$ for $a<0$. 
		
		It is enough to prove that $s_\alpha=H_\alpha|_{t=0}$ for any $(r,s)$-partition $\alpha$.
		
		We consider the ring $\mathbb{Q}[x_{-s}^{\pm1},\cdots,x_{-1}^{\pm1},x_{1}^{\pm1},\cdots,x_{r}^{\pm1}]$. For any integer sequence $\beta=(\beta_{-s},\cdots,\beta_{-1},\beta_1,\cdots,\beta_r)$, we denote by $x^\beta=x_{-s}^{\beta_{-s}}\cdots x_{-1}^{\beta_{-1}}x_{1}^{\beta_{1}}\cdots x_{r}^{\beta_{r}}$ and $R_{ij}x^\beta=x^{R_{ij}\beta}$ for $i<j$. Then the following holds for any $(r,s)$-partition $\alpha$:
		\begin{align}
			\label{eq:schur-transform}
			\sum_{\omega\in \mathfrak{S}_{r+s}}\epsilon(\omega) x^{\alpha+\delta-\omega(\delta)}&=x^{\alpha+\delta}\prod_{i<j}(x_i^{-1}-x_j^{-1})
			=x^\alpha\prod_{i<j}(1-x_ix_j^{-1})
			=\prod_{i<j}(1-R_{ij})x^\alpha.
		\end{align}
		
		It is obvious that $\{x^\beta\mid \beta=(\beta_{-s},\cdots,\beta_{-1},\beta_1,\cdots,\beta_r)\in\Z^{r+s}\}$ is a $\Q$-basis of the ring $\mathbb{Q}[x_{-s}^{\pm1},\cdots,x_{-1}^{\pm1},x_{1}^{\pm1},\cdots,x_{r}^{\pm1}]$. 
		Then there is a $\mathbb{Q}$-linear map $$\psi:\mathbb{Q}[x_{-s}^{\pm1},\cdots,x_{-1}^{\pm1},x_{1}^{\pm1},\cdots,x_{r}^{\pm1}]\rightarrow \mathcal{D}\La,\;x^\beta\mapsto h_\beta.$$ 
		Note that $h_\beta=0$ if $\beta$ is not a $(r,s)$-composition.
		Applying the linear map $\psi$ to \eqref{eq:schur-transform}, we obtain that
		\[
		\sum_{\omega\in \mathfrak{S}_{r+s}}\epsilon(\omega) h_{\alpha+\delta-\omega(\delta)}=\prod_{i<j}(1-R_{ij})h_\alpha.
		\]
		The proof is completed.
	\end{proof}
	
	%
	The following Pieri rule for Schur-Laurent functions is an immediate corollary by setting $t=0$ in Theorem~\ref{thm:Pieri-Ho}. (It can also be proved directly as for Theorem~\ref{thm:Pieri-Ho}, using a well-known and simpler mirror identities for $s_\la$, cf., e.g., \cite[(8)]{Tam12}.)
	
	\begin{corollary} 
		\label{cor:Pieri-Schur}
		For $r\ge 1$ and any partitions $\rho,\nu$, we have
		\begin{align}
			s_{\rho,\nu}  \cdot s^+_r
			&=\sum_{a+b=r} \sum_{\mu \stackrel{a}{\rightarrow} \nu ,\rho\stackrel{b}{\rightarrow} \la}   s_{\la,\mu},
			\label{Pieri:schur3}
		\\
			s_{\rho,\nu}  \cdot s^-_r
			&=\sum_{a+b=r} \sum_{\nu \stackrel{a}{\rightarrow} \mu ,\la\stackrel{b}{\rightarrow} \rho} \,  s_{\la,\mu}.
		\end{align}
	\end{corollary}
	
	The case for $r=1$ of \eqref{Pieri:schur3} can also be deduced from \cite[Theorem 6.1]{SV15}, the Pieri rule for Jack-Laurent symmetric functions, by setting $k=-1$ there.

	\section{Derived Hall algebra of the Jordan quiver}
	\label{sec:Jordan}
	
	For any (essentially small) abelian category $\cb$, we denote by $\Iso(\cb)$ the set of isomorphism classes $[M]$ of objects $M\in\cb$. For a set $S$, we denote by $|S|$ its cardinality. Let $\bfk=\mathbb F_q$ be a finite field of $q$ elements. 
	In this section, we denote by $\QJ$ the Jordan quiver,  i.e., the quiver with a single vertex $1$ and a single loop $\alpha:1\rightarrow 1$.
	Let $\ca={\mathrm {rep}}^\mathrm{nil}_\bfk(\QJ)$ be the category formed by finite-dimensional nilpotent $\bfk$-linear representations. It is well known that the Euler form of $\ca$ is trivial.
	
	The derived Hall algebra of ${\mathrm {rep}}^\mathrm{nil}_\bfk(\QJ)$ is given in \cite{CLR25}. We shall formulate the basic properties including two Pieri rules for the derived Hall algebra of the Jordan quiver. We also establish an algebra isomorphism from the derived Hall algebra to the ring of double symmetric functions.
	
	\subsection{Root category}
	
	It is well known that $\ca=\rep_\bfk^{\rm nil}(\QJ)$ is a uniserial category. Let $S$ be the simple object in $\ca$.
	Then any indecomposable object of $\rep_\bfk^{\rm nil}(\QJ)$  (up to isomorphisms) is of the form $S^{(n)}$ of length $n\geq1$.
	Thus the set of isomorphism classes $\ca$ is canonically isomorphic to the set $\cp$ of all partitions, via the assignment
	\begin{align}
		\la =(\la_1,\la_2,\cdots,\la_r)\mapsto S^{(\la)}= S^{(\la_1)}\oplus \cdots \oplus S^{(\la_r)}.
	\end{align}

	Let $\cd^b(\ca)$ be the bounded derived category of $\ca$ with $[1]$ the shift functor. It is well known that $\cd^b(\ca)$
	is a triangulated category. Let $F$ be an automorphism of $\cd^b(\ca)$. The orbit category $\cd^b(\ca)/F$ has the same objects as $\cd^b(\ca)$ and 
	\begin{align*}
		\Hom_{\cd^b(\ca)/F}(M^\bullet,N^\bullet)=\bigoplus_{i\in\Z}\Hom_{\cd^b(\ca)}(M^\bullet,F^iN^\bullet).
	\end{align*}
	For any $m\ge1$, the orbit category $\cd_m(\ca):=\cd^b(\ca)/[m]$ is called the $m$-periodic derived category of $\ca$, and 
	it is a triangulated category such that the natural projection $\cd^b(\ca)\rightarrow \cd_m(\ca)$ is a triangulated functor; see \cite{PX97,Ke05}. We also denote the shift functor of $\cd_m(\ca)$ by $[1]$. 
	For $m=2$, we also denote $\mathcal{R}(\ca):=\cd_2(\ca)=\cd^b(\ca)/[2]$, which is also called the root category. 
	
	Any object of $\ca$ can be viewed as a stalk complex concentrated at degree zero. This induces a full embedding $\ca\rightarrow \cd^b(\ca)$, and then a full embedding $\ca\rightarrow \cR(\ca)$. We always view $\ca$ as a full subcategory of $\cR(\ca)$ in this way. 
	
	For any object $M^\bullet\in\cR(\ca)$, it is isomorphic to $M_0\oplus M_1[1]$ for some (unique up to isomorphism) $M_0,M_1\in\ca$; see \cite{PX97}. In this way, we call $M_i$ the $i$-th homology group of $M^\bullet$, and denoted by $H^i(M^\bullet)$ for $i=0,1$.

	\subsection{Derived Hall algebra}
	For two objects $M^\bullet$ and $N^\bullet$ in $\cR(\ca)$, denote by $\Aut(M^\bullet)$ the automorphism group of $M^\bullet$, and 
	\begin{align}
		\{M^\bullet,N^\bullet\}:=|\Hom_{\ca}(H^0(M^\bullet),H^0(N^\bullet))|\cdot|\Ext^1_{\ca}(H^1(M^\bullet),H^1(N^\bullet))|.
	\end{align}
	If $M^\bullet=M_0\oplus M_1[1]$ and $N^\bullet=N_0\oplus N_1[1]$ for $M_0,M_1,N_0,N_1\in \ca \subset \mathcal{R}(\ca)$, then 
	\begin{align}
		\label{eq:bilinear-stalk}
		\{M^\bullet,N^\bullet\}=|\Hom_{\ca}(M_0,N_0)||\Ext^1_{\ca}(M_1,N_1)|{=|\Hom_{\ca}(M_0,N_0)||\Hom_{\ca}(M_1,N_1)|}.
	\end{align}
	
	For any $M^\bullet$, $N^\bullet$, $L^\bullet$ in $\cR(\ca)$, define
	\begin{align*}
		\Hom(M^\bullet,N^\bullet)_{L^\bullet}&=
		\{l:M^\bullet\rightarrow N^\bullet\mid \cone(l)\cong L^\bullet\}.
	\end{align*}
	and then the derived Hall number to be 
	\begin{align}
		\label{eq:derhallnum}
		G_{M^\bullet N^\bullet}^{L^\bullet}:=\dfrac{|\Hom(M^\bullet,N^\bullet[1])_{L^\bullet[1]}|}{\{M^\bullet,N^\bullet\}}.
	\end{align}
	From \cite[\S5]{CLR25}, the (Drinfeld dual) derived Hall algebra $\cd\widetilde{\ch}(\bfk \QJ)$ is an associative algebra with $\Q$-basis $\{{[M^\bullet]}\mid M^\bullet\in\cR(\ca)\}$ with the multiplication defined by
	\begin{align}
		{[M^\bullet]}*{[N^\bullet]}=\sum_{[L^\bullet]}G_{M^\bullet N^\bullet}^{L^\bullet} \;{[L^\bullet]},
	\end{align}
	and the identity ${[0]}$.
	
	Given objects $M,N,L\in\ca$, let $\Ext^1(M,N)_L\subseteq \Ext^1(M,N)$ be the subset parameterizing extensions whose middle term is isomorphic to $L$. The {\em Hall algebra} (or {\em Ringel-Hall algebra}) $\widetilde{\ch}(\bfk\QJ)$ is defined to be the $\Q$-vector space with the isoclasses $[M]$ of objects $M \in \ca$ as a basis and multiplication given by (cf., e.g., \cite{Br13})
	\begin{align}
		\label{eq:mult}
		[M]* [N]=\sum_{[L]\in \Iso(\ca)}G_{MN}^{L}\cdot[L],
	\end{align}
	where $G_{MN}^{L}=\dfrac{|\Ext^1(M,N)_L|}{|\Hom(M,N)|}$.

	Given three objects $M,N,L$, the Hall number is defined to be
	$$F_{MN}^L:= |\{X\subseteq L\mid X \cong N\text{ and }L/X\cong M\}|.$$
	Denote by $\aut(M)$ the automorphism group of $M$. The Riedtmann-Peng formula reads
	\begin{align}
		\label{eq:RP}
		{F_{MN}^L}= \frac{|\Ext^1(M,N)_L|}{|\Hom(M,N)|} \cdot \frac{|\aut(L)|}{|\aut(M)|\cdot |\aut(N)|}.
	\end{align}
	It is well known that $\widetilde{\ch}(\bfk\QJ)$ is commutative. 
	
	Let $\widetilde{\ch}(\bfk\QJ)\otimes\widetilde{\ch}(\bfk\QJ)$ be the natural tensor algebra. Then we have the following.
	
	\begin{lemma}[\text{\cite[Propositions 5.1,5.2]{CLR25}}]
		\label{lem:commDHA}
		The derived Hall algebra $\cd\widetilde{\ch}(\bfk \QJ)$ is commutative.
		Furthermore, we have an algebra isomorphism $\th(\bfk\QJ)\otimes\th(\bfk\QJ)\cong \cd\th(\bfk\QJ)$.
	\end{lemma}

	\begin{proposition}
		\label{generator}
		The derived Hall algebra $\cd\widetilde{\ch}(\bfk \QJ)$ is isomorphic to 
		\begin{enumerate}
			\item
			the polynomial $\Q$-algebra in the infinitely many generators
			$[S^{(1^r)}], [S^{(1^r)}[1]]$  for $r\ge 1$;
			\item
			the polynomial $\Q$-algebra in the infinitely many generators
			$ [S^{(r)}], [S^{(r)}[1]]$ for $r\ge 1$.
		\end{enumerate}
	\end{proposition}
	
	\begin{proof}
		We only prove the first statement, and the second one can be proved in the same way.
		
		From \cite[Chapter II, (2.3)]{Mac95}, $\th(\bfk \QJ)$ is isomorphic to the polynomial algebra in the infinitely many generators $[S^{(1^r)}],$ for $r\ge 1$. Then the desired result follows from Lemma \ref{lem:commDHA}.
	\end{proof}
	
	We shall provide an applicable formula for derived Hall numbers by using Hall numbers. To do this, we need to make use of the category of $2$-periodic complexes.
	
	A $2$-periodic complex over $\ca$ is $$(\xymatrix{ M^0 \ar@<0.5ex>[r]^{d^0}& M^1 \ar@<0.5ex>[l]^{d^1}}),\quad d^1d^0=d^0d^1=0.$$
	Let $\cc_2(\ca)$ be the category of $2$-periodic complexes of $\ca$. It is well known that 
	$$|\Hom_{\mathcal{R}(\ca)}(M^\bullet,N^\bullet[1])_{L^\bullet[1]}|= \sum_{\substack{[X^\bullet]\in\Iso(\cc_2(\ca)),\\
			H^i(X^\bullet)\cong L_i}}
	|\Ext_{\mathcal{C}_2(\ca)}(M^\bullet,N^\bullet)_{X^\bullet}|.$$ 
	
	Recall the semi-derived Ringel-Hall algebra $\cs\cd\ch(\ca)$ of $\cc_2(\ca)$ defined in \cite{LP21}.  Let $\ch(\cc_2(\ca))$ be the Hall algebra of $\cc_2(\ca)$, and $I$ be its ideal generated by all $[M^\bullet]-[N^\bullet\oplus K^\bullet]$ if there is a short exact sequence $0\rightarrow K^\bullet\rightarrow M^\bullet\rightarrow N^\bullet\rightarrow0$, with $K^\bullet$ acyclic. Then $\cs\cd\ch(\ca)$ is the localization of the quotient algebra $\ch(\cc_2(\ca))/I$ such that $[K^\bullet]$ are invertible for any acyclic complexes $K^\bullet$. 
	
	For any object $M\in\ca$, we define the following objects in $\cc_{\Z_2}(\ca)$:
	\begin{align}
		\label{stalks}
		\begin{split}
			K_M:=&(\xymatrix{ M \ar@<0.5ex>[r]^{1}& M \ar@<0.5ex>[l]^{0}  }),\qquad \,\, K_M^*:=(\xymatrix{ M \ar@<0.5ex>[r]^{0}& M \ar@<0.5ex>[l]^{1}  }),
			\\
			C_M:=&(\xymatrix{ 0 \ar@<0.5ex>[r]& M \ar@<0.5ex>[l]  }),\qquad \quad C_M^*:=(\xymatrix{ M\ar@<0.5ex>[r]& 0 \ar@<0.5ex>[l]  }).
		\end{split}
	\end{align}

	Let $K_0(\ca)$ be the Grothendieck group of $\ca$. Denote by $\widehat{M}$ the class of $M\in\ca$ in $K_0(\ca)$. For any $\alpha\in\ca$, we denote $[K_\alpha]:=[K_M]*[K_N]^{-1}$ if $\alpha=\widehat{M}-\widehat{N}$ in $\cs\cd\ch(\ca)$; $[K_\alpha^*]$ is defined similarly. By definition, we know 
	\begin{align}\label{sdh1}
		[M^\bullet]*[N^\bullet]=\sum_{\substack{[L_0],[L_1]\in\Iso(\ca),\\\alpha,\beta\in K_0(\ca)}}&\sum_{\substack{[X^\bullet]\in\Iso(\cc_2(\ca)),\\
				H^i(X^\bullet)\cong L_i,\widehat{\Im(d^0)}=\alpha,\widehat{\Im(d^1)}=\beta}}
		\\\notag
		&\frac{|\Ext_{\mathcal{C}_2(\ca)}(M^\bullet,N^\bullet)_{X^\bullet}|}{|\Hom(M^\bullet,N^\bullet)|} [C_{L_0}\oplus C_{L_1}^*]*[K_{\alpha}]*[K^*_{\beta}].
	\end{align}
	Here $d^0,d^1$ are the differentials of $N^\bullet$. 
	
	From \cite[Proposition 3.3,Example 2.2]{LR24}, by noting the Euler form is trivial, we know
	\begin{align}\label{sdh2}
		[M^\bullet]*[N^\bullet]=&\sum_{\substack{[L_0],[L_1]\in\Iso(\ca),\\\alpha,\beta\in K_0(\ca)}}\sum_{[B_0],[B_1],[C_0],[C_1]}\sum_{\substack{[A_0],[A_1],\\ \widehat{A_0}=\alpha,\widehat{A_1}=\beta}}
		\\\notag&\times\dfrac{|\aut(A_0)||\aut(A_1)||\aut(B_0)||\aut(B_1)||\aut(C_0)||\aut(C_1)|}{|\aut(L_0)||\aut(L_1)|}
		\\\notag
		&\times{F_{A_0B_0}^{M_0}F_{B_0C_0}^{L_0}F_{C_0A_1}^{N_0}F_{A_1B_1}^{M_1}F_{B_1C_1}^{L_1}F_{C_1A_0}^{N_1}}
		[C_{L_0}\oplus C_{L_1}^*]*[K_{\alpha}]*[K^*_{\beta}].
	\end{align}
	
	\begin{proposition}
		\label{prop:dHallvsHall}
		For any $M^\bullet=M_0\oplus M_1[1],N^\bullet=N_0\oplus N_1[1],L^\bullet=L_0\oplus L_1[1]\in \cR(\ca)$ with $M_i,N_i,L_i\in\ca$ for $i=0,1$, we have 
		{
			\begin{align}
				\label{six-hall}
				G_{M^\bullet N^\bullet}^{L^\bullet}
				=&\sum_{\substack{[A_0],[A_1],[B_0],[B_1],\\ [C_0],[C_1]\in \mathrm{Iso}(\ca)}}{F_{A_0B_0}^{M_0}F_{B_0C_0}^{L_0}F_{C_0A_1}^{N_0}F_{A_1B_1}^{M_1}F_{B_1C_1}^{L_1}F_{C_1A_0}^{N_1}}
				\\\notag&\qquad\times\dfrac{|\aut(A_0)||\aut(A_1)||\aut(B_0)||\aut(B_1)||\aut(C_0)||\aut(C_1)|}{|\aut(L_0)||\aut(L_1)|}.
			\end{align}
		}
	\end{proposition}

	\begin{proof}
		
		
		By comparing formulas \eqref{sdh1} and \eqref{sdh2}, we know
		\begin{align*}&|\Hom_{\mathcal{R}(\ca)}(M^\bullet,N^\bullet[1])_{L^\bullet[1]}|\\
			&= \sum_{\substack{[X^\bullet]\in\Iso(\cc_2(\ca)),\\
					H^i(X^\bullet)\cong L_i}}
			|\Ext_{\mathcal{C}_2(\ca)}(M^\bullet,N^\bullet)_{X^\bullet}|
			\\
			&=|\Hom(M^\bullet,N^\bullet)|\sum_{\substack{[A_0],[A_1],[B_0],[B_1],\\ [C_0],[C_1]\in \mathrm{Iso}(\ca)}}{F_{A_0B_0}^{M_0}F_{B_0C_0}^{L_0}F_{C_0A_1}^{N_0}F_{A_1B_1}^{M_1}F_{B_1C_1}^{L_1}F_{C_1A_0}^{N_1}}
			\\\notag&\qquad\times\dfrac{|\aut(A_0)||\aut(A_1)||\aut(B_0)||\aut(B_1)||\aut(C_0)||\aut(C_1)|}{|\aut(L_0)||\aut(L_1)|}.
		\end{align*}
		Then the desired formula follows since $|\Hom(M^\bullet,N^\bullet)|=\{M^\bullet,N^\bullet\}$; see \eqref{eq:bilinear-stalk}.
	\end{proof}

	
	\subsection{Generic derived Hall algebra}
	\label{subsec:g-derived}
	According to Hall and Steinitz (cf. \cite[III]{Mac95}), there exist unique polynomials $F_{\mu,\nu}^\lambda(T),G_{\mu,\nu}^\lambda(T)\in\Z[T]$ such that
	$$F_{S^{(\mu)},S^{(\nu)}}^{S^{(\lambda)}}=F_{\mu,\nu}^\lambda(q),\qquad G_{S^{(\mu)},S^{(\nu)}}^{S^{(\lambda)}}=G_{\mu,\nu}^\lambda(q).$$

	For any partition $\lambda=(\lambda_1,\lambda_2,\ldots)$, let $n(\la)=\sum_i(i-1)\la_i$.
	We denote 
	\begin{align}
		\label{eq:phim}
		\varphi_{m}(t)&=(1-t)(1-t^2)\cdots(1-t^m),\quad \forall m\in\N;
		\\
		\label{def:bla}
		b_\la(t) &= \prod_{i\geq1} \varphi_{m_i(\la)}(t),\quad \forall \lambda\in\cp.
	\end{align}
	The following is well known (see, e.g., \cite[Lemma~ 2.8]{Sch06}): 
	\begin{align}
		\label{eq:aut}
		|\Aut(S^{(\lambda)})|=&\bq^{|\lambda|+2n(\lambda)}\prod_i(1-\bq^{-1})(1-\bq^{-2})\cdots (1-\bq^{-m_i(\lambda)})
		\\\notag
		=&\bq^{|\lambda|+2n(\lambda)}\prod_{i\geq1}\varphi_{m_i(\lambda)}(\bq^{-1})=\bq^{|\lambda|+2n(\lambda)}b_\lambda(\bq^{-1}).
	\end{align}
	
	\begin{lemma}
		\label{lem:Hallpoly}
		For partitions $\lambda,\tilde{\la},\mu,\tilde{\mu},\nu,\tilde{\nu}$, there exists a Laurent polynomial $\mathcal G_{\mu,\tilde{\mu},\nu,\tilde{\nu}}^{\lambda,\tilde{\la}}(T)\in \Z[T,T^{-1}]$ such that
		\begin{align}
			[S^{(\mu)}\oplus S^{(\tilde{\mu})}[1]]*[S^{(\nu)}\oplus S^{(\tilde{\nu})}[1]]=&
			\sum_{\lambda,\tilde{\la}} \mathcal G_{\mu,\tilde{\mu},\nu,\tilde{\nu}}^{\lambda,\tilde{\la}}(q)[S^{(\lambda)}\oplus S^{(\tilde{\la})}]
		\end{align}
	\end{lemma}
	
	\begin{proof}
		By Proposition \ref{prop:dHallvsHall} and the above, we can write
		\begin{align}
			[S^{(\mu)}\oplus S^{(\tilde{\mu})}[1]]*[S^{(\nu)}\oplus S^{(\tilde{\nu})}[1]]=&
			\sum_{\lambda,\tilde{\la}} \mathcal G_{\mu,\tilde{\mu},\nu,\tilde{\nu}}^{\lambda,\tilde{\la}}(q)[S^{(\lambda)}\oplus S^{(\tilde{\la})}],
		\end{align}
		where $\mathcal G_{\mu,\tilde{\mu},\nu,\tilde{\nu}}^{\lambda,\tilde{\la}}(T)=\frac{F_1(T)}{F_2(T)}$ for  $F_1(T),F_2(T)\in\Z[T]$, and $F_2(\bq)=|\aut(S^{(\la)})||\aut(S^{(\tilde{\la})})|$.
		A comparison with \eqref{eq:derhallnum} shows that $\bq^m G_{\mu,\tilde{\mu},\nu,\tilde{\nu}}^{\lambda,\tilde{\la}}(q)=\frac{\bq^m F_1(q)}{F_2(q)}\in\Z$ for some $m\in \N$.  Therefore, $F_2(T)$ divides $T^m F_1(T)$, and thus we have $\mathcal G_{\mu,\nu}^{\lambda}(T)\in \Z[T,T^{-1}]$ by noting that $\bq$ can be any prime power.
	\end{proof}
	
	
	We now define the {\em generic derived Hall algebra} of the Jordan quiver over $\Q(t)$, and denote it by $\cd\th(\QJ)$. The algebra $\cd\th(\QJ)$ is the free $\Q(t)$-module with a basis $\{\fu_{\lambda,\tilde{\la}}\mid \lambda,\tilde{\la}\in\cp\}$ and multiplication given by
	\begin{align}
		\label{eq:g}
		\fu_{\mu,\tilde{\mu}}*\fu_{\nu,\tilde{\nu}}=\sum_{\lambda,\tilde{\la}\in\cp}\mathcal G_{\mu,\tilde{\mu},\nu,\tilde{\nu}}^{\lambda,\tilde{\la}}(t^{-1})\fu_{\lambda,\tilde{\la}}.
	\end{align}

	\subsection{Horizontal Pieri rules} 
	
	We now establish the first version of the Pieri rules in $\cd\th(\bfk \QJ)$.
	
	\begin{lemma}
		\label{lem:HallfomrulaJordan}
		For $r\geq 1$ and any partitions $\rho,\nu$, the following identities hold in $\cd\th(\bfk \QJ)$:
		\begin{align}
			\label{Pieri:h}
			{[S^{(\rho)}\oplus S^{(\nu)}[1]]}*{[S^{(r)}]}
			&= \sum_{a=0}^r\sum_{\la,\mu}F_{\rho,(a)}^{\la}(q)F_{(r-a),\mu}^\nu(q)
			\\\notag
			&\quad\quad \times\dfrac{|\Aut(S^{(\rho)})|\cdot|\Aut(S^{(a)})|\cdot|\Aut(S^{(r-a)})|}{|\Aut(S^{(\la)})|}\cdot {[S^{(\la)}\oplus S^{(\mu)}[1]]}
			,
			\\
			\label{Pieri:v}
			{[S^{(\rho)}\oplus S^{(\nu)}[1]]}*{[S^{(1^r)}]}
			&= \sum_{a=0}^r\sum_{\la,\mu}F_{\rho,(1^a)}^{\la}(q)F_{(1^{r-a}),\mu}^\nu(q)\\\notag
			&\quad\quad\times\dfrac{|\Aut(S^{(\rho)})|\cdot|\Aut(S^{(1^{r})})|}{|\Aut(S^{(\la)})|\cdot q^{a(r-a)}}\cdot {[S^{(\la)}\oplus S^{(\mu)}[1]]}.
		\end{align}
		Dually, we have 
		\begin{align}
			\label{Pieri:h2}
			{[S^{(\rho)}\oplus S^{(\nu)}[1]]}*{[S^{(r)}[1]]}
			&= \sum_{a=0}^r\sum_{\la,\mu}F_{\la,(a)}^{\rho}(q)F_{(r-a),\nu}^\mu(q)
			\\\notag
			&\quad\quad \times\dfrac{|\Aut(S^{(\nu)})|\cdot|\Aut(S^{(a)})|\cdot|\Aut(S^{(r-a)})|}{|\Aut(S^{(\mu)})|}\cdot {[S^{(\la)}\oplus S^{(\mu)}[1]]}
			,
			\\
			\label{Pieri:v2}
			{[S^{(\rho)}\oplus S^{(\nu)}[1]]}*{[S^{(1^r)}[1]]}
			&= \sum_{a=0}^r\sum_{\la,\mu}F_{\la,(1^a)}^{\rho}(q)F_{(1^{r-a}),\nu}^\mu(q)
			\\
			\notag&\qquad\times\dfrac{|\Aut(S^{(\nu)})|\cdot|\Aut(S^{(1^{r})})|}{|\Aut(S^{(\mu)})|\cdot q^{a(r-a)}}\cdot {[S^{(\la)}\oplus S^{(\mu)}[1]]}	.
		\end{align}
		
	\end{lemma}
	
	\begin{proof}
		We only prove identities \eqref{Pieri:h} and \eqref{Pieri:v}. Using Proposition \ref{prop:dHallvsHall}, we have
		\begin{align*}
			{[S^{(\rho)}\oplus S^{(\nu)}[1]]}*{[S^{(r)}]}=&\sum_{a=0}^r\sum_{\la,\mu}F_{\rho,(a)}^{\la}(q)F_{(a),(r-a)}^{(r)}(q)F_{(r-a),\mu}^\nu(q)
			\\\notag&\quad\times\dfrac{|\Aut(S^{(\rho)})|\cdot|\Aut(S^{(a)})|\cdot|\Aut(S^{(r-a)})|}{|\Aut(S^{(\la)})|}\cdot {[S^{(\la)}\oplus S^{(\mu)}[1]]},
		\end{align*}
		where $F_{(a),(r-a)}^{(r)}$  equals $1$ for any $0\leq a\leq r$. Hence, we obtain the formula \eqref{Pieri:h}. 
		
		Similarly, by Proposition \ref{prop:dHallvsHall} and Riedtmann-Peng formula, we get
		\begin{align*}
			&{[S^{(\rho)}\oplus S^{(\nu)}[1]]}*{[S^{(1^r)}]}
			\\&=\sum_{a=0}^r\sum_{\la,\mu}F_{\rho,(1^a)}^{\la}(q)F_{(1^a),(1^{r-a})}^{(1^r)}(q)F_{(1^{r-a}),\mu}^\nu(q)\dfrac{|\Aut(S^{(\rho)})|\cdot|\Aut(S^{(1^a)})|\cdot|\Aut(S^{(1^{r-a})})|}{|\Aut(S^{(\la)})|} \\
			&\qquad\qquad\times{[S^{(\la)}\oplus S^{(\mu)}[1]]}	
			\\&=\sum_{a=0}^r\sum_{\la,\mu}F_{\rho,(1^a)}^{\la}(q)F_{(1^{r-a}),\mu}^\nu(q)\frac{|\Ext(S^{(1^a)}),S^{(1^{r-a})})_{S^{(1^r)})}|}{|\Hom(S^{(1^a)}),S^{(1^{r-a})})|}\cdot\dfrac{|\Aut(S^{(\rho)})|\cdot|\Aut(S^{(1^{r})})|}{|\Aut(S^{(\la)})|} \\
			&\qquad\qquad\times{[S^{(\la)}\oplus S^{(\mu)}[1]]}	
			\\&=\sum_{a=0}^r\sum_{\la,\mu}F_{\rho,(1^a)}^{\la}(q)F_{(1^{r-a}),\mu}^\nu(q)\dfrac{|\Aut(S^{(\rho)})|\cdot|\Aut(S^{(1^{r})})|}{|\Aut(S^{(\la)})|\cdot q^{a(r-a)}}\cdot {[S^{(\la)}\oplus S^{(\mu)}[1]]}.	
		\end{align*}
		The proof is completed.
	\end{proof}
	
	We shall now renormalize the formulas \eqref{Pieri:h}--\eqref{Pieri:v2} to make them compatible with the symmetric function side.
	
	For a horizontal $r$-strip $\sigma=\nu-\mu$, let $I=I_{\nu-\mu}$ be the set of integers $i\geq 1$ such that $\sigma_i'=1$ and $\sigma_{i+1}'=0$.
	Recall from \cite[II,(4.13)]{Mac95} that
	\begin{align*}
		F_{\mu,(r)}^{\nu}(T)
		=& \frac{T^{n(\nu)-n(\mu)}}{1-T^{-1}}\prod\limits_{i\in I_{\nu-\mu}}(1-T^{-m_i(\nu)})
	\end{align*}
	if $r>0$; and $F_{\mu,(0)}^{\nu}(T)=1$. So we have
	\begin{align}
		\label{eq:Gmurnu}
		F_{\mu,(r)}^{\nu}(\bq)\cdot|\aut(S^{(r)})|=\bq^{n(\nu)-n(\mu)+r}\prod\limits_{i\in I_{\nu-\mu}}(1-\bq^{-m_i(\nu)}),\quad \forall r\geq0.
	\end{align}
	
	For any partitions $\lambda,\mu$, denote by
	\begin{align}
		\label{eq:haV}
		\iVh_{\lambda,\mu}:= \bq^{-|\lambda|-n(\lambda)-n(\mu)}\cdot {[S^{(\lambda)}\oplus S^{(\mu)}[1]]}\in \cd\th(\bfk \QJ)
	\end{align}
	and \[
	\iVh^+_{\la}=\iVh_{\la,\emptyset},\qquad\iVh^-_{\mu}=\iVh_{\emptyset,\mu}.
	\]
	In particular, for any $r>0$,
	\begin{align}
		\label{eq:haV-hv}
		\begin{split}
			&\iVh^+_{(r)}= \bq^{-r}\cdot {[S^{(r)}]},\qquad\iVh^+_{(1^r)}= \bq^{-\binom{r+1}{2}}\cdot {[S^{(1^r)}]},\\
			&\iVh^-_{(r)}=  {[S^{(r)}[1]]}, \qquad\quad\,\,\iVh^-_{(1^r)}= \bq^{-\binom{r}{2}}\cdot {[S^{(1^r)}[1]]}.
		\end{split}
	\end{align}
	
	\begin{remark}
		We can also define
		{
			\begin{align}
				\label{eq:haV2}
				{\iVh}_{\lambda,\mu}':= \bq^{-|\mu|-n(\lambda)-n(\mu)}\cdot {[S^{(\lambda)}\oplus S^{(\mu)}[1]]}\in \cd\th(\bfk \QJ)
			\end{align}
		}for any partitions $\lambda,\mu$, and it is direct to check that there is an
		automorphism of $\cd\th(\bfk \QJ)$ given by mapping $\iVh_{\la,\mu}$ to ${\iVh}_{\lambda,\mu}'$.
	\end{remark}
	
	
	
	\begin{proposition}
		\label{prop:PieriHall-Ho}
		For $r\ge 1$ and any partitions $\rho,\nu$, we have{
			\begin{align}
				\label{eq:haPH-ho}
				\iVh_{\rho,\nu}* \iVh^+_{(r)}= &
				\sum_{a+b=r} \sum_{\mu \stackrel{a}{\rightarrow} \nu ,\rho\stackrel{b}{\rightarrow} \la} \, \varphi_{\nu/\mu}(q^{-1}) \psi_{\la/\rho}(q^{-1}) \;  \iVh_{\la,\mu},
			\end{align}
			\begin{align}
				\label{eq:haPH-ho2}
				\iVh_{\rho,\nu}* \iVh^-_{(r)}= &
				\sum_{a+b=r} \sum_{\nu \stackrel{a}{\rightarrow} \mu ,\la\stackrel{b}{\rightarrow} \rho} \, \psi_{\mu/\nu}(q^{-1}) \varphi_{\rho/\la}(q^{-1}) \;  \iVh_{\la,\mu}.
			\end{align}
		}
	\end{proposition}
	
	\begin{proof}
		We only prove the formula \eqref{eq:haPH-ho}. By \eqref{eq:haV}, the identity \eqref{Pieri:h} translates into 
		\begin{align*}
			\iVh_{\rho,\nu}* \iVh^+_{(r)}
			= \sum_{a=0}^r\sum_{\la,\mu}&F_{\rho,(a)}^{\la}(q)F_{(r-a),\mu}^\nu(q)\cdot \bq^{|\lambda|+n(\lambda)+n(\mu)-|\rho|-n(\rho)-n(\nu)-r}
			\\
			& \times\dfrac{|\Aut(S^{(\rho)})|\cdot|\Aut(S^{(a)})|\cdot|\Aut(S^{(r-a)})|}{|\Aut(S^{(\la)})|}\cdot\iVh_{\la,\mu}.   
		\end{align*}
		Then combining \eqref{eq:Gmurnu} and $|\mathrm{Aut}(S^{(\la)})|=q^{|\la|+2n(\la)}b_\la(q^{-1})$, we get
		\begin{align*}
			\iVh_{\rho,\nu}* \iVh^+_{(r)}=&
			\sum_{a+b=r} \sum_{\mu \stackrel{a}{\rightarrow} \nu ,\rho\stackrel{b}{\rightarrow} \la} \, \varphi_{\nu/\mu}(q^{-1}) \varphi_{\la/\rho}(q^{-1})\cdot q^{n(\la)-n(\rho)+n(\nu)-n(\mu)+r} \;  
			\\ &\qquad\times\bq^{|\lambda|+n(\lambda)+n(\mu)-|\rho|-n(\rho)-n(\nu)-r}\frac{q^{|\rho|+2n(\rho)}b_\rho(q^{-1})}{q^{|\la|+2n(\la)}b_\la(q^{-1})}\cdot\iVh_{\la,\mu}
			\\=&\sum_{a+b=r} \sum_{\mu \stackrel{a}{\rightarrow} \nu ,\rho\stackrel{b}{\rightarrow} \la} \, \varphi_{\nu/\mu}(q^{-1}) \varphi_{\la/\rho}(q^{-1})\frac{b_\rho(q^{-1})}{b_\la(q^{-1})}\cdot \iVh_{\la,\mu}.
		\end{align*}
		The desired formula \eqref{eq:haPH-ho} immediately follows by \cite[III, (5.12)]{Mac95}
		\[
		\psi_{\la/\rho}(q^{-1})=\varphi_{\la/\rho}(q^{-1})\cdot\frac{b_\rho(q^{-1})}{b_\la(q^{-1})}.
		\]The proof is completed.
	\end{proof}

	\subsection{Derived Hall algebra and the ring $\mathcal{D}\La$}
	
	Recall the  ring $\mathcal{D}\La_{t} = \Q(t) [v_1^\pm, v_2^\pm, \ldots]$ and $\mathcal{D}\La_{q^{-1}} = \Q [v_1^\pm, v_2^\pm, \ldots]$, where $V_{\la,\mu}\in \mathcal{D}\La_{q^{-1}}$ is defined by specializing $V_{\la,\mu}\in\mathcal{D}\La_{t}$ at $t=q^{-1}$ for given $\bfk=\mathbb F_q$. Recall $\iVh_{\la,\mu}$ from \eqref{eq:haV}.
	
	\begin{theorem}
		\label{thm:iso}
		There exists a $\Q$-algebra isomorphism $\Phi_{(\bq)}: \cd\th(\bfk \QJ) \longrightarrow \mathcal{D}\La_{q^{-1}}$ such that
		\begin{align*}
			\Phi_{(\bq)} ({[S^{(r)}]}) =q^r v^+_r ,\qquad\Phi_{(\bq)} ({[S^{(r)}[1]]}) = v^-_r \quad \text{  for  }r\ge 1.
		\end{align*}
		Moreover, for any partitions $\la,\mu$, we have
		\begin{align}
			\label{eqn:Phibasis}
			\Phi_{(\bq)}({[S^{(\la)}\oplus S^{(\mu)}[1]]})=\bq^{|\la| +n(\la)+n(\mu)} V_{\la,\mu},
			\qquad
			\Phi_{(\bq)}(\iVh_{\la,\mu})= V_{\la,\mu}.
		\end{align}
	\end{theorem}
	
	\begin{proof}
		The first statement follows from Proposition \ref{generator}.
		
		The second statement will be proved using Proposition \ref{prop:PieriHall-Ho} and Theorem~ \ref{thm:Pieri-Ho}. First we consider the case $\la=\emptyset$ or $\mu=\emptyset$. We have 
		\[
		\iVh_{\rho,\emptyset}* \iVh^+_{(r)}= 
		\sum_{\rho\stackrel{r}{\rightarrow} \la}  \psi_{\la/\rho}(q^{-1}) \;  \iVh_{\la,\emptyset},\quad
		\iVh_{\emptyset,\nu}* \iVh^-_{(r)}=\sum_{\nu \stackrel{r}{\rightarrow} \mu } \, \psi_{\mu/\nu}(q^{-1})  \;  \iVh_{\emptyset,\mu}.
		\]
		Then it follows from \cite[III, (3.4)]{Mac95} that \[
		\Phi_{(\bq)}(\iVh_{\la,\emptyset})= V_{\la,\emptyset},\quad \Phi_{(\bq)}(\iVh_{\emptyset,\mu})= V_{\emptyset,\mu},
		\]
		On this basis, we prove \eqref{eqn:Phibasis} by induction on $n=|\la|$ and the dominance of $\la$ on $\mathcal{P}_n$ for $\iVh_{\la,\mu}$. The case for $n=0$ is clear. If $n>1$, let $\rho$ be the partition obtained from $\la$ by deleting the last row (may be empty). Suppose that the last row has $r$ elements. Then by Proposition \ref{prop:PieriHall-Ho}, we have
		\[
		\iVh_{\rho,\mu}* \iVh^+_{(r)}=\iVh_{\la,\mu}+  \sum_{{\rho\stackrel{r}{\rightarrow} \tilde{\la},\la\triangleright\tilde{\la}}} \,  \psi_{\la/\rho}(q^{-1}) \;  \iVh_{\tilde{\la},\mu}+
		\sum_{\substack{a+b=r\\a\ne0}} \sum_{{\tilde{\mu} \stackrel{a}{\rightarrow} \mu ,\rho\stackrel{b}{\rightarrow} \tilde{\la}}} \, \varphi_{\mu/\tilde{\mu}}(q^{-1}) \psi_{\tilde{\la}/\rho}(q^{-1}) \;  \iVh_{\tilde{\la},\tilde{\mu}}.
		\]
		By Theorem~ \ref{thm:Pieri-Ho}, we have
		\[
		V_{\rho,\mu}\cdot v^+_{r}=V_{\la,\mu}+  \sum_{{\rho\stackrel{r}{\rightarrow} \tilde{\la},\la\triangleright\tilde{\la}}} \,  \psi_{\la/\rho}(q^{-1}) \;  V_{\tilde{\la},\mu}+
		\sum_{\substack{a+b=r\\a\ne0}} \sum_{{\tilde{\mu} \stackrel{a}{\rightarrow} \mu ,\rho\stackrel{b}{\rightarrow} \tilde{\la}}} \, \varphi_{\mu/\tilde{\mu}}(q^{-1}) \psi_{\tilde{\la}/\rho}(q^{-1}) \;  V_{\tilde{\la},\tilde{\mu}}.
		\]
		By the inductive assumption and comparing the above two identities, we obtain that $\Phi_{(\bq)}(\iVh_{\la,\mu})= V_{\la,\mu}$. 
	\end{proof}
	
	\begin{remark}
		\label{rem:PQ}
		The Hall products in this paper and in \cite{Mac95} use different normalizations (though the resulting Hall algebras of the Jordan quiver are isomorphic via $[S^\la] \mapsto [S^\la]/ |\Aut(S^\la)|$).
	\end{remark}
	
	As a corollary, we can formulate the isomorphism for the generic derived Hall algebra:
	\begin{corollary}
		\label{cor:isomorphism}
		There exists a $\Q(t)$-algebra isomorphism $\Phi_t: \cd\th(\QJ) \longrightarrow \mathcal{D}\La_t$ such that
		\begin{align*}
			\Phi_t (\fu_{(r),\emptyset}) =t^{-r} v^+_r ,\qquad\Phi_t (\fu_{\emptyset,(r)}) = v^-_r \quad \text{  for  }r\ge 1.
		\end{align*}
		Moreover, for any partitions $\la,\mu$, we have{
			\begin{align*}
				\Phi_t(\fu_{\lambda,\mu})=t^{-|\la| -n(\la)-n(\mu)} V_{\la,\mu},
				\qquad
				\Phi_t(\iVh_{\la,\mu})= V_{\la,\mu}.
		\end{align*}}
	\end{corollary}

	\subsection{Vertical Pieri rules}
	
	Let $n, r \in \Z$. Define
	\begin{align*}\begin{bmatrix} n \\ r \end{bmatrix}_+(t)=\frac{\varphi_n(t)}{\varphi_r(t)\varphi_{n-r}(t)}
		=\frac{(1-t^n)(1-t^{n-1})\cdots(1-t^{n-r+1})}{(1-t)(1-t^2)\cdots(1-t^r)}
	\end{align*}
	for $r\ge 0$ and
	$\begin{bmatrix} n \\ r \end{bmatrix}_+(t)=0$
	for $r<0$.
	
	For any partitions $\la,\mu$ and $m\geq0$, define (cf. \cite[III, (3.2)]{Mac95})
	\begin{align}
		\label{eq:fmu}
		f_{\mu,(1^m)}^\la(t)=&\begin{cases}
			\prod_{i\geq1} \qbinom{\la'_i-\la'_{i+1}}{\la'_i-\mu'_i}_+(t),& \text{ if }\lambda-\mu \text{ is a vertical }m\text{-strip},
			\\
			0,& \text{ otherwise.}
		\end{cases}
	\end{align}
	
	Recall that, $n(\la)=\sum_{i\ge1}(i-1)\la_i$ for any partition $\la$. By \cite[III, (3.3)]{Mac95}, we have
	\begin{align}
		\label{eq:Gf}
		F^\la_{\mu,(1^m)}(q)=q^{n(\la)-n(\mu)-n(1^m)} f^\la_{\mu,(1^m)}(q^{-1}).
	\end{align}
	Given partitions $\mu, \nu$ and $a \in \N$, we use
	\[
	\mu {\downarrow}\nu,
	\qquad (\text{respectively, } \mu \stackrel{a}{\downarrow}\nu)
	\]
	to denote that $\nu \leq \mu$ and $\mu -\nu$ is a vertical strip (respectively, a vertical $a$-strip).
	
	\begin{theorem}
		\label{prop:haPieri-v}
		For $r\geq1$ and any partitions $\rho,\nu$, we have in $\cd\widetilde{\ch}(\bfk\QJ)${
			\begin{align}
				\label{eq:haPieri-v1}
				\iVh_{\rho,\nu}*\iVh^+_{(1^r)}
				=&\sum_{a+b=r} \sum_{\la\stackrel{a}{\downarrow} \rho} \sum_{\nu \stackrel{b}{\downarrow}\mu}
				\ \frac{b_{\rho}(\bq^{-1})}{b_\la(\bq^{-1})}\varphi_r(\bq^{-1})  f_{\rho,(1^{a})}^\la(\bq^{-1}) f_{\mu,(1^b)}^\nu(\bq^{-1}) \cdot \iVh_{\la,\mu},
			\end{align}
			\begin{align}
				\label{eq:haPieri-v2}
				\iVh_{\rho,\nu}*\iVh^-_{(1^r)}
				=&\sum_{a+b=r} \sum_{\rho\stackrel{a}{\downarrow} \la} \sum_{\mu \stackrel{b}{\downarrow}\nu}
				\  \frac{b_{\nu}(\bq^{-1})}{b_\mu(\bq^{-1})}\varphi_r(\bq^{-1})  f_{\la,(1^{a})}^\rho(\bq^{-1}) f_{\nu,(1^b)}^\mu(\bq^{-1}) \cdot \iVh_{\la,\mu}.
		\end{align}}
	\end{theorem}
	
	\begin{proof}
		We only prove the first equation.
		Starting from \eqref{eq:haV} and \eqref{Pieri:v} and using also \eqref{eq:Gf} and $|\mathrm{Aut}(S^{(\la)})|=q^{|\la|+2n(\la)}b_\la(q^{-1})$, we have
		\begin{align*}
			\iVh_{\rho,\nu}*\iVh^+_{(1^r)}=&\sum_{a+b=r} \sum_{\la\stackrel{a}{\downarrow} \rho} \sum_{\nu \stackrel{b}{\downarrow}\mu}F_{\rho,(1^a)}^{\la}F_{(1^{b}),\mu}^\nu\cdot\dfrac{|\Aut(S^{(\rho)})|\cdot|\Aut(S^{(1^{r})})|}{|\Aut(S^{(\la)})|\cdot q^{a(r-a)}}
			\\&\qquad\qquad\qquad\times q^{n(\la)+n(\mu)+|\la|-n(\rho)-n(\nu)-|\rho|-n(1^r)-r} \cdot\iVh_{\la,\mu}
			\\=&\sum_{a+b=r} \sum_{\la\stackrel{a}{\downarrow} \rho} \sum_{\nu \stackrel{b}{\downarrow}\mu}f_{\rho,(1^a)}^{\la}(q^{-1})f_{(1^{b}),\mu}^\nu(q^{-1})\cdot\dfrac{|\Aut(S^{(\rho)})|\cdot|\Aut(S^{(1^{r})})|}{|\Aut(S^{(\la)})|\cdot q^{a(r-a)}}
			\\&\qquad\qquad\qquad\times q^{2n(\la)+|\la|-2n(\rho)-|\rho|-n(1^r)-r-n(1^a)-n(1^b)} \cdot\iVh_{\la,\mu}
			\\=&\sum_{a+b=r} \sum_{\la\stackrel{a}{\downarrow} \rho} \sum_{\nu \stackrel{b}{\downarrow}\mu}f_{\rho,(1^a)}^{\la}(q^{-1})f_{(1^{b}),\mu}^\nu(q^{-1})\cdot\dfrac{b_\rho(q^{-1})}{b_\la(q^{-1})}
			\\&\qquad\qquad\qquad\times|\Aut(S^{(1^{r})})|\cdot q^{-n(1^r)-r-n(1^a)-n(1^{r-a})-a(r-a)} \cdot\iVh_{\la,\mu}
			\\=&\sum_{a+b=r} \sum_{\la\stackrel{a}{\downarrow} \rho} \sum_{\nu \stackrel{b}{\downarrow}\mu}
			\frac{b_{\rho}(\bq^{-1})}{b_\la(\bq^{-1})}\varphi_r(\bq^{-1})  f_{\rho,(1^{a})}^\la(\bq^{-1}) f_{\mu,(1^b)}^\nu(\bq^{-1}) \cdot \iVh_{\la,\mu},
		\end{align*}
		which is the desired \eqref{eq:haPieri-v1}.
	\end{proof}

	\section{Generating functions and  their transition relations}
	\label{sec:generating}
	
	In this section, we work with the derived Hall algebra of the Jordan quiver $\cd\widetilde{\ch}(\bfk \QJ) \otimes_{\Q} \Q(\sqq)$, where $\bfk =\mathbb F_q$ and $\sqq=\sqrt{q}.$ Define the following generating functions of several natural generating sets in the derived Hall algebra (compare \cite{BKa01, Sch06, LRW23, LRW25}):
	
	\begin{align}\label{gen1}
		\widetilde{E}(y,z)=\sum_{r,t\ge0}\sqq^{r(r-1)+t(t-1)}\frac{[S^{(1^r)}\oplus S^{(1^t)}[1]]}{|\Aut(S^{(1^r)})||\Aut(S^{(1^t)})|}y^r z^t,
	\end{align}
	
	\begin{align}
		\widetilde{H}(y,z)=\sum_{r,t\ge0}\sum_{\la\vdash r,\mu \vdash t}\frac{[S^{(\la)}\oplus S^{(\mu)}[1]]}{|\Aut(S^{(\la)})||\Aut(S^{(\mu)})|}y^r z^t,
	\end{align}
	
	\begin{align}
		\widetilde{\Theta}(y,z)=\sum_{r,t\ge0}{[S^{(r)}\oplus S^{(t)}[1]]}y^r z^t,
	\end{align}

	The remainder of this section is devoted to determining the transition relations among these generating functions.

	\subsection{Euler's identity}
	
	In this subsection, $q$ can also be interpreted as a formal variable.
	
	For $n\in \N \cup \{\infty \}$, we denote by
	\[
	(x;q)_n =(1-x)(1-qx)\ldots (1-q^{n-1} x).
	\]
	We have the following identity due to Euler:
	\begin{align}
		\label{eq:Euler}
		\exp_q\Big(\frac{x}{1-q} \Big):= \sum_{r\geq0} \frac{x^r}{(1-q)(1-q^2)\ldots (1-q^r)}
		= \frac{1}{(x; q)_\infty}.
	\end{align}
	It is elementary to rewrite $\dfrac{1}{(x; q)_\infty}$ in terms of the usual exponential function:
	\[
	\frac{1}{(x; q)_\infty} = \exp \Big (\sum_{k\ge 1} \frac{x^k}{k (1-q^k)} \Big).
	\]
	Hence we have
	\begin{align}  \label{exp=exp}
		\exp_q \big(\frac{x}{1-q} \big) = \exp \Big (\sum_{k\ge 1} \frac{x^k}{k (1-q^k)} \Big).
	\end{align}
	
	\subsection{$\widetilde{E}(y,z)$ vs $\widetilde{H}(y,z)$}
	We first formulate the connection between generating functions {$\widetilde{H}(y,z)$ and $\widetilde{E}(y,z)$}.
	
	\begin{proposition}
		\label{e-h}
		We have
		\begin{align}
			\widetilde{H}(y,z)\widetilde{E}(-y,-z)=\big(\mathrm{exp}_q(\frac{yz}{1-q})\big)^2.    
		\end{align}
	\end{proposition}
	
	\begin{proof}
		Using \eqref{six-hall}, we obtain
		\begin{align*}
			&\widetilde{H}(y,z)\widetilde{E}(-y,-z)
			\\&= \sum_{a,b\ge0}\sum_{r,t\ge0}\sum_{\la\vdash r,\mu\vdash t}(-1)^{a+b}\sqq^{a(a-1)+b(b-1)}\frac{[S^{(1^a)}\oplus S^{(1^b)}[1]]}{|\Aut(S^{(1^a)})||\Aut(S^{(1^b)})|}
			\\&\qquad*  \frac{[S^{(\la)}\oplus S^{(\mu)}[1]]}{|\Aut(S^{(\la)})||\Aut(S^{(\mu)})|} y^{a+r} z^{b+t}
			\\&= \sum_{a,b\ge0}\sum_{r,t\ge0}\sum_{\la\vdash r,\mu\vdash t}(-1)^{a+b}\sqq^{a(a-1)+b(b-1)}\sum_{d=0}^b\sum_{c=0}^a \sum_{\rho,\nu,\alpha,\beta\in\cp} F_{(1^d), \alpha}^{\la}(q)F_{\alpha, (1^{a-c})}^{\rho}(q)F_{(1^{a-c}), (1^c)}^{(1^a)}(q)
			\\
			&\qquad \times F_{(1^c), \beta}^{\mu}(q)F_{\beta, (1^{b-d})}^{\nu}(q)F_{(1^{b-d}), (1^d)}^{(1^b)}(q)
			\dfrac{|\Aut(S^{(1^d)})||\Aut(S^{(\alpha)})||\Aut(S^{(1^{a-c})})|}{|\Aut(S^{(\rho)})||\Aut(S^{(\nu)})|} \\
			&\qquad \times|\Aut(S^{(1^{c})})||\Aut(S^{(\beta)})||\Aut(S^{(1^{b-d})})|\cdot [S^{(\rho)}\oplus S^{(\nu)}[1]] y^{a+r} z^{b+t}.
		\end{align*}
		
		Due to the Riedtmann-Peng formula, we deduce that\[
		\sum_{\la\vdash r}F_{(1^{d}),\alpha}^\la\frac{|\Aut(S^{({\alpha})})|}{|\Aut(S^{({\la})})|}=\frac{1}{|\Aut(S^{({1^d})})|},\quad\sum_{\mu\vdash t}F_{(1^{c}),\beta}^\mu\frac{|\Aut(S^{({\beta})})|}{|\Aut(S^{({\mu})})|}=\frac{1}{|\Aut(S^{({1^c})})|}.
		\]
		Hence, 
		\begin{align*}
			&\widetilde{H}(y,z)\widetilde{E}(-y,-z)
			\\
			&= \sum_{a,b\ge0}\sum_{r,t\ge0}(-1)^{a+b}\sqq^{a(a-1)+b(b-1)}\sum_{d=0}^b\sum_{c=0}^a \sum_{\rho,\nu}\sum_{\alpha\vdash (r-d)}\sum_{\beta\vdash (t-c)} F_{\alpha, (1^{a-c})}^{\rho}(q)F_{(1^{a-c}), (1^c)}^{(1^a)}(q)\\
			&\quad\times F_{\beta, (1^{b-d})}^{\nu}(q)F_{(1^{b-d}),(1^d)}^{(1^b)}(q)
			\dfrac{|\Aut(S^{(1^{a-c})})||\Aut(S^{(1^{b-d})})|}{|\Aut(S^{(1^a)})||\Aut(S^{(1^b)})||\Aut(S^{(\rho)})||\Aut(S^{(\nu)})|} \\
			&\quad \times [S^{(\rho)}\oplus S^{(\nu)}[1]] y^{a+r} z^{b+t}
			\\=& \sum_{a,b,r,t\ge0}(-1)^{a+b}\sqq^{a(a-1)+b(b-1)}\sum_{d=0}^b\sum_{c=0}^a \sum_{\rho,\nu}\sum_{\alpha\vdash (r-d)}\sum_{\beta\vdash (t-c)} F_{\alpha, (1^{a-c})}^{\rho}(q)F_{\beta ,(1^{b-d})}^{\nu}(q) \sqq^{c(a-c)+d(b-d)}
			\\
			&\quad \times\begin{bmatrix}
				a\\c
			\end{bmatrix}_\sqq \begin{bmatrix}
				b\\d
			\end{bmatrix}_\sqq
			\dfrac{|\Aut(S^{(1^{a-c})})||\Aut(S^{(1^{b-d})})|}{|\Aut(S^{(1^a)})||\Aut(S^{(1^b)})||\Aut(S^{(\rho)})||\Aut(S^{(\nu)})|} [S^{(\rho)}\oplus S^{(\nu)}[1]] y^{a+r} z^{b+t}.
		\end{align*}
		In the last equality, we use the formulas
		\[
		F_{(1^{a-c}),(1^c)}^{(1^a)}(q)=\sqq^{c(a-c)}\begin{bmatrix}
			a\\c
		\end{bmatrix}_\sqq,\quad F_{(1^{b-d}),(1^d)}^{(1^b)}(q)=\sqq^{d(b-d)}\begin{bmatrix}
			b\\d
		\end{bmatrix}_\sqq.
		\]
		
		Fix $\rho,\nu$ and $a+r=e,b+t=f$ in the following. The proof is divided into the following three cases.
		
		Case \underline{$\rho=\nu=\emptyset$}. Then $\alpha=\beta=\emptyset$ and $b=d=r,a=c=t$. The coefficient of $y^{e}z^{e}$ is equal to
		\begin{align*}
			&\sum_{c+d=e}\frac{(-1)^{c+d}\sqq^{c(c-1)+d(d-1)} }{|\Aut(S^{(1^{c})})||\Aut(S^{(1^{d})})|}
			=\sum_{c+d=e}\frac{1}{(1-q)\cdots(1-q^c)}\cdot\frac{1}{(1-q)\cdots(1-q^d)}.
		\end{align*}
		
		Case \underline{$\rho\ne \emptyset$}. 
		We also fix $c,d$. Then $F_{\alpha ,(1^{a-c})}^\rho(q)\ne0$ implies that $|\rho|+c+d=a+r$ and $c\le a\le \ell(\rho)+c$. Therefore, the coefficient of $\dfrac{[S^{(\rho)}\oplus S^{(\nu)}[1]]}{|\Aut(S^{({\rho})})||\Aut(S^{({\nu})})|}y^{e}z^{f}$ is equal to 
		
		\begin{align*}
			& \sum_{a+r=e}\sum_{b+t=f}(-1)^{a+b}\sqq^{a(a-1)+b(b-1)}\sum_{\alpha\vdash (r-d)}\sum_{\beta\vdash (t-c)} F_{\alpha, (1^{a-c})}^{\rho}(q)F_{\beta ,(1^{b-d})}^{\nu}(q) 
			\\&\qquad\qquad\times\sqq^{c(a-c)+d(b-d)}\begin{bmatrix}
				a\\c
			\end{bmatrix}_\sqq \begin{bmatrix}
				b\\d
			\end{bmatrix}_\sqq\cdot \dfrac{|\Aut(S^{(1^{a-c})})||\Aut(S^{(1^{b-d})})|}{|\Aut(S^{(1^a)})||\Aut(S^{(1^b)})|} 
			\\&=\sum_{b+t=f}(-1)^b \sqq^{b(b-1)}\sqq^{d(b-d)}\begin{bmatrix}
				b\\d
			\end{bmatrix}_\sqq \sum_{\beta}F^\nu_{\beta, (1^{b-d})}(q)\frac{|\Aut(S^{(1^{b-d})})|}{|\Aut(S^{(1^{b})})|}
			\\&\qquad\qquad\times\sum_{c\le a\le \ell(\rho)+c}(-1)^a \sqq^{a(a-1)}\sqq^{c(a-c)}\begin{bmatrix}
				a\\c
			\end{bmatrix}_\sqq \sum_{\alpha}F^\rho_{\alpha ,(1^{a-c})}\frac{|\Aut(S^{(1^{a-c})})|}{|\Aut(S^{(1^{a})})|}
			\\&=0,
		\end{align*}
		where the last equality follows from
		\begin{align*}
			\sum_{c\le a\le \ell(\rho)+c}(-1)^a \sqq^{a(a-1)}\sqq^{c(a-c)}\begin{bmatrix}
				a\\c
			\end{bmatrix}_\sqq \sum_{\alpha}F^\rho_{\alpha ,(1^{a-c})}\frac{|\Aut(S^{(1^{a-c})})|}{|\Aut(S^{(1^{a})})|}=0,
		\end{align*}
		deduced from the proof of \cite[Proposition 6.3]{LRW25}.
		
		Case \underline{$\nu\ne\emptyset$}. It is similar to the second case, hence omitted here.
		
		Therefore, we have \begin{align*}
			\widetilde{H}(y,z)\widetilde{E}(-y,-z)&=\sum_{e\ge0}\sum_{c+d=e}\frac{1}{(1-q)\cdots(1-q^c)}\cdot\frac{1}{(1-q)\cdots(1-q^d)}y^e z^e
			\\\notag&=\big(\mathrm{exp}_q(\frac{yz}{1-q})\big)^2. 
		\end{align*}
		The proof is completed.
	\end{proof}
	
	\subsection{$\widetilde{E}(y,z)$ or $\widetilde{H}(y,z)$ vs $\widetilde{\Theta}(y,z)$}
	
	We need some preparation before formulating the relations among $\widetilde{E}(y,z)$, $\widetilde{H}(y,z)$ and $\widetilde{\Theta}(y,z)$.
	
	By the formula \eqref{six-hall}, we obtain that
	\begin{align*}
		&[S^{(r)}\oplus S^{(t)}[1]]*[S^{(1^a)}\oplus S^{(1^b)}[1]]
		\\
		&=\sum_{d=0,1}\sum_{c=0,1} \sum_{\rho,\nu,\alpha,\beta\in\cp} F_{(1^d), \alpha}^{(r)}(q)F_{\alpha ,(1^{a-c})}^{\rho}(q)F_{(1^{a-c}),( 1^c)}^{(1^a)}(q)
		F_{(1^c), \beta}^{(t)}(q)F_{\beta ,(1^{b-d})}^{\nu}(q)F_{(1^{b-d}), (1^d)}^{(1^b)}(q)
		\\
		&\qquad\times\dfrac{|\Aut(S^{(1^d)})||\Aut(S^{(\alpha)})||\Aut(S^{(1^{a-c})})||\Aut(S^{(1^{c})})||\Aut(S^{(\beta)})||\Aut(S^{(1^{b-d})})|}{|\Aut(S^{(\rho)})||\Aut(S^{(\nu)})|}\\
		&\qquad \times [S^{(\rho)}\oplus S^{(\nu)}[1]].
	\end{align*}
	Note that both $c$ and $d$ take values in $\{0,1\}$. For fixed $c$ and $d$, both $\rho$ and $\nu$ admit exactly two choices, which arise respectively from the trivial and nontrivial extensions of the pairs $(S^{(\alpha)}, S^{(1^{a-c})})$ and $(S^{(\beta)}, S^{(1^{b-d})})$. To systematically enumerate all possible cases, we present Table \ref{tab1}.
	
	\begin{table}
		\begin{tabular}{|c|c|c|c|c|c|}
			\hline
			&$c$ & $d$ & $\rho$ & $\nu$& The coefficient of $[S^{(\rho)}\oplus S^{(\nu)}[1]]$\\\hline
			a)&$0$ & $0$ & $(r,1^a)$ & $(t,1^b)$ 
			&$q^{-a-b}$
			\\
			\hline
			b)&$0$ & $0$ & $(r+1,1^{a-1})$ & $(t,1^b)$ 
			&$(1-q^{-a})q^{-b}$\\
			\hline
			c)&$0$ & $0$ & $(r,1^a)$ & $(t+1,1^{b-1})$ 
			&$q^{-a}(1-q^{-b})$
			\\
			\hline
			d)&$0$ & $0$ & $(r+1,1^{a-1})$ & $(t+1,1^{b-1})$ 
			&$(1-q^{-a})\cdot(1-q^{-b})$
			\\
			\hline
			e)&$1$ & $0$ & $(r,1^{a-1})$ & $(t-1,1^b)$ 
			&$(q^{a}-1)q^{1-a-b}$\\\hline
			f)&$1$ & $0$ & $(r+1,1^{a-2})$ & $(t-1,1^b)$ 
			&$(q^{a}-1)(1-q^{1-a})q^{-b}$\\\hline
			g)&$1$ & $0$ & $(r,1^{a-1})$ & $(t,1^{b-1})$ 
			&$(q^{a}-1)q^{1-a}(1-q^{-b})$\\\hline
			h)&$1$ & $0$ & $(r+1,1^{a-2})$ & $(t,1^{b-1})$ 
			&$(q^{a}-1)(1-q^{1-a})(1-q^{-b})$\\\hline
			i)&$0$ & $1$ & $(r-1,1^{a})$ & $(t,1^{b-1})$ 
			&$(q^{b}-1)q^{1-a-b}$\\\hline
			j)&$0$ & $1$ & $(r,1^{a-1})$ & $(t,1^{b-1})$ 
			&$(q^{b}-1)q^{1-b}(1-q^{-a})$\\\hline
			j)&$0$ & $1$ & $(r-1,1^{a})$ & $(t+1,1^{b-2})$ 
			&$(q^{b}-1)q^{-a}(1-q^{1-b})$\\\hline
			l)&$0$ & $1$ & $(r,1^{a-1})$ & $(t+1,1^{b-2})$ 
			&$(q^{b}-1)(1-q^{-a})(1-q^{1-b})$\\\hline
			m)&$1$ & $1$ & $(r-1,1^{a-1})$ & $(t-1,1^{b-1})$ 
			&$(q^{b}-1)(q^{a}-1)q^{2-a-b}$\\\hline
			n)&$1$ & $1$ & $(r,1^{a-2})$ & $(t-1,1^{b-1})$ 
			&$(q^{b}-1)(q^{a}-1)q^{1-b}(1-q^{1-a})$\\\hline
			o)&$1$ & $1$ & $(r-1,1^{a-1})$ & $(t,1^{b-2})$ 
			&$(q^{b}-1)(q^{a}-1)q^{1-a}(1-q^{1-b})$\\\hline
			p)&$1$ & $1$ & $(r,1^{a-2})$ & $(t,1^{b-2})$ 
			&$(q^{b}-1)(q^{a}-1)(1-q^{1-a})(1-q^{1-b})$\\
			\hline
		\end{tabular}
		\vspace{0.5cm}
		\caption{Computation of $[S^{(r)}\oplus S^{(t)}[1]]*[S^{(1^a)}\oplus S^{(1^b)}[1]]$}
		\label{tab1}
	\end{table}

	\begin{proposition}
		\label{theta-e-h}
		We have
		\begin{align}
			\label{theta-e}
			\widetilde{\Theta}(y,z)&=(1-qyz)^2\frac{\widetilde{E}(-y,-z)}{\widetilde{E}(-qy,-qz)},
			\\
			\label{theta-h}
			\widetilde{\Theta}(y,z)&={(1-yz)^{-2}}\frac{\widetilde{H}(qy,qz)}{\widetilde{H}(y,z)}.
		\end{align}
	\end{proposition}
	
	\begin{proof}
		First, let us prove \eqref{theta-e}. By definition, we have the following formula: 
		\begin{align*}
			&\widetilde{\Theta}(y,z)\widetilde{E}(-qy,-qz)\\&=\sum_{a,b\ge 0}\sum_{r,t\ge 0}(-q)^{a+b}\sqq^{a(a-1)+b(b-1)}\frac{[S^{(r)}\oplus S^{(t)}[1]]*[S^{(1^a)}\oplus S^{(1^b)}[1]]}{|\Aut(S^{(1^{a})})||\Aut(S^{(1^b)})|}y^{a+r}z^{b+t}.
		\end{align*}
		Observe that only the terms $[S^{(r,1^{a})}\oplus S^{(t,1^{b})}[1]]y^{a+r+k}z^{b+t+k}$, for $k=0,1,2$
			can appear on the right-hand side which come from a)--d), e)--l) and m)--p) in Table \ref{tab1}, respectively.
			Moreover, their coefficients are equal to $0$ if $r>1$ or $t>1$ by a simple computation.
		Hence it remains to calculate  coefficients of $[S^{(1^a)}\oplus S^{(1^b)}[1]]y^{a+k} z^{b+k}$ for $k=0,1,2$, which comes from $[S^{(i)}\oplus S^{(j)}[1]]*[S^{(1^{a+k-i})}\oplus S^{(1^{b+k-j})}[1]]$ for certain $(i,j)$'s.
            
            If $k=0$,
             then $(i,j)=(0,0),(1,0), (0,1)$ or $(1,1)$.
        So its coefficient equals to
		\begin{align*}
			&(-q)^{a+b}\sqq^{a(a-1)+b(b-1)}\frac{1}{|\Aut(S^{(1^{a})})||\Aut(S^{(1^b)})|}\\&
			+(-q)^{a+b-1}\sqq^{(a-1)(a-2)+b(b-1)}\frac{q^{1-a}}{|\Aut(S^{(1^{a-1})})||\Aut(S^{(1^b)})|}
			\\&+(-q)^{a+b-1}\sqq^{(a-1)a+(b-1)(b-2)}\frac{q^{1-b}}{|\Aut(S^{(1^{a})})||\Aut(S^{(1^{b-1})})|}
			\\&+(-q)^{a+b-2}\sqq^{(a-1)(a-2)+(b-1)(b-2)}\frac{q^{2-a-b}}{|\Aut(S^{(1^{a-1})})||\Aut(S^{(1^{b-1})})|}
			\\&=(-1)^{a+b}\frac{\sqq^{a(a-1)+b(b-1)}}{{|\Aut(S^{(1^{a})})||\Aut(S^{(1^b)})|}}.
		\end{align*}
		If $k=2$,
        then $(i,j)=(1,1),(2,1), (1,2)$ or $(2,2)$.
       So its coefficient  equals to
		\begin{align*}
			&(-q)^{a+b+2}\sqq^{a(a+1)+b(b+1)}\frac{(q^{a+1}-1)(q^{b+1}-1)}{{|\Aut(S^{(1^{a+1})})||\Aut(S^{(1^{b+1})})|}}
			\\&+(-q)^{a+b+1}\sqq^{a(a-1)+b(b+1)}\frac{(q-q^{1-a})(q^{b+1}-1)}{{|\Aut(S^{(1^{a})})||\Aut(S^{(1^{b+1})})|}}
			\\&+(-q)^{a+b+1}\sqq^{a(a+1)+b(b-1)}\frac{(q^{a+1}-1)(q-q^{1-b})}{{|\Aut(S^{(1^{a+1})})||\Aut(S^{(1^{b})})|}}
			\\&+(-q)^{a+b}\sqq^{a(a-1)+b(b-1)}\frac{(q-q^{1-a})(q-q^{1-b})}{{|\Aut(S^{(1^{a})})||\Aut(S^{(1^{b})})|}}
			\\&=(-1)^{a+b+2}\dfrac{q^2\sqq^{a(a-1)+b(b-1)}}{{|\Aut(S^{(1^{a})})||\Aut(S^{(1^b)})|}}.
		\end{align*}
		If $k=1$, 
        then $(i,j)=(0,1),(1,1), (0,2)$ or $(1,2)$ while $c=1, d=0$; and $(i,j)=(1,0),(1,1), (2,0)$ or $(2,1)$ while $c=0, d=1$ in Table \ref{tab1}, respectively.
        By a similar calculation, its coefficient is equal to
		\begin{align*}
			(-1)^{a+b+1}\frac{2q\sqq^{a(a-1)+b(b-1)}}{{|\Aut(S^{(1^{a})})||\Aut(S^{(1^b)})|}}.
		\end{align*}
		Therefore, we have
		\begin{align*}
			\widetilde{\Theta}(y,z)\widetilde{E}(-qy,-qz)&=\sum_{a,b\ge0} (-1)^{a+b}\frac{\sqq^{a(a-1)+b(b-1)}}{{|\Aut(S^{(1^{a})})||\Aut(S^{(1^b)})|}}[S^{(1^a)}\oplus S^{(1^b)}[1]]y^az^b
			\\&\quad+\sum_{a,b\ge0} (-1)^{a+b+2}\frac{q^2\sqq^{a(a-1)+b(b-1)}}{{|\Aut(S^{(1^{a})})||\Aut(S^{(1^b)})|}}[S^{(1^a)}\oplus S^{(1^b)}[1]]y^{a+2}z^{b+2}
			\\&\quad+\sum_{a,b\ge0} (-1)^{a+b+1}\frac{2q\sqq^{a(a-1)+b(b-1)}}{{|\Aut(S^{(1^{a})})||\Aut(S^{(1^b)})|}}[S^{(1^a)}\oplus S^{(1^b)}[1]]y^{a+1}z^{b+1}
			\\&=(1+q^2y^2z^2-2qyz)\widetilde{E}(-y,-z)
			\\&=(1-qyz)^2\widetilde{E}(-y,-z).
		\end{align*}
		
		For \eqref{theta-h}, we make use of the Proposition \ref{e-h}.
		Hence, we deduce that
		\begin{align*}
			\widetilde{\Theta}(y,z)&=(1-qyz)^2\frac{\widetilde{E}(-y,-z)}{\widetilde{E}(-qy,-qz)}\\&=
			(1-qyz)^2\frac{\widetilde{E}(-y,-z)}{\widetilde{E}(-qy,-qz)}\frac{\widetilde{H}(y,z)}{\widetilde{H}(qy,qz)}\frac{\widetilde{H}(qy,qz)}{\widetilde{H}(y,z)}
			\\&=
			(1-qyz)^2{(\exp_q{(\frac{yz}{1-q})})^2}({\exp_q{(\frac{q^2yz}{1-q})})^{-2}}
			\frac{\widetilde{H}(qy,qz)}{\widetilde{H}(y,z)}
			\\&=
			(1-qyz)^2(1-yz)^{-2}(1-qyz)^{-2}
			\frac{\widetilde{H}(qy,qz)}{\widetilde{H}(y,z)}
			\\&=
			(1-yz)^{-2}
			\frac{\widetilde{H}(qy,qz)}{\widetilde{H}(y,z)},
		\end{align*}
		where in the fourth equation we use the equation\[
		\exp_q{(\frac{x}{1-q})}=(1-x)^{-1}\exp_q{(\frac{qx}{1-q})}=(1-x)^{-1}(1-qx)^{-1}\exp_q{(\frac{q^2x}{1-q})}.
		\]
		The proof is completed.
	\end{proof}

	\subsection{Comparison with classic generating functions}
	
	In this subsection, we shall compare 
	the generating function  $\widetilde{\Theta}(y,z)$, $\widetilde{E}(y,z)$, $\widetilde{H}(y,z)$ in $\cd\ch(\bfk \QJ)$ and the ones  in the classical Hall algebra.
	
	We also consider $\widetilde{\ch}(\bfk \QJ)$ over $\Q(\sqq)$. 
	\begin{lemma}
		\label{lem:embeddings}
		There exist $\Q(\sqq)$-algebra embeddings 
		\begin{align}
			\iota^+:\widetilde{\ch}(\bfk \QJ)\longrightarrow\cd\widetilde{\ch}(\bfk \QJ),\qquad  \iota^-:\widetilde{\ch}(\bfk \QJ)\longrightarrow\cd\widetilde{\ch}(\bfk \QJ)
		\end{align}
		{such that $\iota^+([M])=[M]$, $\iota^-([M])=[M[1]]$.}
	\end{lemma}
	
	\begin{proof}
		It is enough to prove that the maps $\iota^\pm$ preserve the multiplications, which follows from Proposition \ref{prop:dHallvsHall}.
	\end{proof}
	
	From Lemma \ref{lem:embeddings}, inspired by the generating functions of the classic Hall algebra $\widetilde{\ch}(\bfk \QJ)$ (cf. \cite[Page 222]{BKa01}),
	we define in $\cd\widetilde{\ch}(\bfk \QJ)$
	\begin{align}
		\label{eq:classic-SFTheta}
		&{\Theta}_1(y)
		=\sum_{r\ge0}{[S^{(r)}]}y^r,& {\Theta}_2(z)
		&=\sum_{t\ge0}{[ S^{(t)}[1]]} z^t;
		\\
			\label{eq:classic-SFE}
		&{E}_1(y)=\sum_{r\ge0}\sqq^{r(r-1)}\frac{[S^{(1^r)}]}{|\Aut(S^{(1^r)})|}y^r ,& {E}_2(z)&=\sum_{t\ge0}\sqq^{t(t-1)}\frac{[S^{(1^t)}[1]]}{|\Aut(S^{(1^t)})|} z^t;
		\\
			\label{eq:classic-SFH}
			&{H}_1(y)=\sum_{r\ge0}\sum_{\la\vdash r}\frac{[S^{(\la)}]}{|\Aut(S^{(\la)})|}y^r,& {H}_2(z)&=\sum_{t\ge0}\sum_{\mu \vdash t}\frac{[S^{(\mu)}[1]]}{|\Aut(S^{(\mu)})|} z^t.
	\end{align}
	We also define (see \cite[Page 23]{Mac95})
	\begin{align}
		\label{def:p}
		\begin{split}
			\widetilde{P}(y)=\sum_{r\ge1}\widetilde{P}_ry^{r-1},\quad \text{where } \widetilde{P}_r=\sum_{\la\vdash r} \varphi_{\ell(\la)-1}(q)\frac{[S^{(\la)}]}{|\Aut(S^{(\la)})|},
			\\
			\widetilde{P}'(z)=\sum_{t\ge1}{\widetilde{P}}'_tz^{t-1},\quad \text{where } \widetilde{P}'_t=\sum_{\mu\vdash t} \varphi_{\ell(\mu)-1}(q)\frac{[S^{(\mu)}[1]]}{|\Aut(S^{(\mu)})|}.
		\end{split}
	\end{align}
	Then we have the following proposition.
	
	\begin{theorem}
		\label{prop:transition}
		We have
		\begin{align}
			\label{eq:Theta-Theta12}
			\widetilde{\Theta}(y,z)=&\Theta_1(y)\Theta_2(z)\exp\big( \sum_{k\ge1}\frac{1-q^k}{k}(yz)^k\big),
			\\
			\widetilde{E}(y,z)=&E_1(y)E_2(z)\exp_q\big( \frac{yz}{1-q}\big),
			\\
			\widetilde{H}(y,z)=&H_1(y)H_2(z)\exp_q\big( \frac{yz}{1-q}\big).
		\end{align}
	\end{theorem}
	
	The proof of Theorem \ref{prop:transition} is given in the next subsection. The strategy is 
	to use $P_1(y)$, $P_2(z)$ to express $\widetilde{\Theta}(y,z)$, $\widetilde{E}(y,z)$, $\widetilde{H}(y,z)$.

	\subsection{The proof of Theorem \ref{prop:transition}}
	
	Denote $\widetilde{\Theta}_{m,n}=\dfrac{[S^{(m)}\oplus S^{(n)}[1]]}{\sqq-\sqq^{-1}}$ in $\cd\widetilde{\ch}(\bfk \QJ)$. It follows that
	\begin{align}
		\label{theta-theta}
		\widetilde{\Theta}(y,z)=(\sqq-\sqq^{-1})\sum_{m,n\ge0}\widetilde{\Theta}_{m,n}y^m z^n.
	\end{align}
	
	Define $\widetilde{T}_{m,n}\in \cd\widetilde{\ch}(\bfk \QJ)$ by
	\begin{align}
		\label{theta-T}
		\widetilde{\Theta}(y,z)=\exp\big((\sqq-\sqq^{-1})\sum_{m,n\ge0}\widetilde{T}_{m,n}y^m z^n\big).
	\end{align}
	Note that $\widetilde{T}_{0,0}=0$.
	
	\begin{lemma}
		We have
		\begin{align}  
			\label{t*theta}
			m\widetilde{\Theta}_{m,n}=(\sqq-\sqq^{-1})\sum_{1\le a\le m}\sum_{0\le b\le n}a\widetilde{T}_{a,b}*\widetilde{\Theta}_{m-a,n-b},
		\end{align}
		and
		\begin{align}        
			n\widetilde{\Theta}_{m,n}=(\sqq-\sqq^{-1})\sum_{0\le a\le m}\sum_{1\le b\le n}b\widetilde{T}_{a,b}*\widetilde{\Theta}_{m-a,n-b}.
		\end{align}
	\end{lemma}
	
	\begin{proof}
		We only prove the first formula. Applying the differentiation $\frac{\partial}{\partial y}$ to \eqref{theta-theta} and \eqref{theta-T}, we have
		\begin{align}\label{DerivationFormula}
			&\sum_{m\ge1,n\geq 0} (\sqq-\sqq^{-1}) m \widetilde{\Theta}_{m,n} y^{m-1}z^n
			\\\notag&= \Big( \sum_{a\ge 1,b\ge0}(\sqq-\sqq^{-1}) a \widehat{T}_{a,b} y^{a-1}z^b \Big)  \Big(\sum_{m,n\geq 0} (\sqq-\sqq^{-1})\widetilde{\Theta}_{m,n} y^m z^n\Big).
		\end{align}
		Now the desired formula follows by comparing coefficients of $y^{m-1}z^n$ on both sides.
	\end{proof}
	
	\begin{proposition}
		\label{T-P}
		We have
		\begin{align} \label{t-p}
			\widetilde{T}_{m,n}=\delta_{n,0}\sqq^{m}\frac{[m]_\sqq}{m}\widetilde{P}_m+ \delta_{m,0}\sqq^{n}\frac{[n]_\sqq}{n}\widetilde{P}'_n-\delta_{m,n}\sqq^m\dfrac{[m]_\sqq}{m}\quad (m\ge1\text{ or }n\ge1),
		\end{align}
		\begin{align}
			\label{theta-p}
			\widetilde{\Theta}(y,z)=\exp\big(\sum_{r\ge1}\frac{q^r-1}{r}\widetilde{P}_ry^r\big)  \exp\big(\sum_{t\ge1}\frac{q^t-1}{t}\widetilde{P}'_tz^t\big)\exp\big( \sum_{k\ge1}\frac{1-q^k}{k}(yz)^k\big).
		\end{align}
	\end{proposition}

	\begin{proof}
		The second formula \eqref{theta-p} follows immediately 
		from \eqref{t-p} and \eqref{theta-T}.
		So it remains to prove \eqref{t-p}, we get it by induction. 
		It is clear for $m=0$ or $n=0$. 
		
		Let $m,n\ge1$ and assume that \eqref{t-p} holds for any $m',n'$ with $m'<m,n'<n$. Without loss of generality, we assume $m\le n$. Then by \eqref{t*theta} and the inductive assumption, we obtain
		\begin{align*}       &m\widetilde{T}_{m,n}\\&=m\widetilde{\Theta}_{m,n}-(\sqq-\sqq^{-1})\sum_{\substack{1\le a\le m,0\le b\le n\\(a,b)\ne(m,n)}} a\widetilde{T}_{a,b}*\widetilde{\Theta}_{m-a,n-b}
			\\&=m\widetilde{\Theta}_{m,n}-(\sqq-\sqq^{-1})\sum_{{1\le a\le m}} a\widetilde{T}_{a,0}*\widetilde{\Theta}_{m-a,n}-(\sqq-\sqq^{-1})\sum_{\substack{1\le a\le \min(m,n)\\(a,a)\ne(m,n)}} a\widetilde{T}_{a,a}*\widetilde{\Theta}_{m-a,n-a}
			\\&=m\widetilde{\Theta}_{m,n}-(\sqq-\sqq^{-1})\sum_{{1\le a\le m}} \sqq^a[a]_\sqq \widetilde{P}_a*\widetilde{\Theta}_{m-a,n}+(\sqq-\sqq^{-1})\sum_{\substack{1\le a\le \min(m,n)\\(a,a)\ne(m,n)}} \sqq^a[a]_\sqq\widetilde{\Theta}_{m-a,n-a}.
		\end{align*}
		
		By definition, we have
		\begin{align*}
			(\sqq-\sqq^{-1})\widetilde{P}_a*\widetilde{\Theta}_{m-a,n}=&
			\sum_{\la\vdash a} \varphi_{\ell(\la)-1}(q)\frac{[S^{(\la)}]}{|\Aut(S^{(\la)})|}*[S^{(m-a)}\oplus S^{(n)}[1]]
			\\=&\sum_{\la\vdash a} \frac{\varphi_{\ell(\la)-1}(q)}{|\Aut(S^{(\la)})|}\sum_{\rho,c,\alpha} F_{(n-c),\alpha}^\la F_{\alpha,(m-a)}^\rho\frac{|\Aut(S^{(n-c)})||\Aut(S^{(\alpha)})|}{|\Aut(S^{(\rho)})|}
			\\
			&\quad\times|\Aut(S^{(m-a)})|[S^{(\rho)}\oplus S^{(c)}[1]]
			\\=&\sum_{\la\vdash a} {\varphi_{\ell(\la)-1}(q)}\sum_{\rho,c,\alpha} \frac{|\Ext^1(S^{(\alpha)},S^{(n-c)})_{S^{(\la)}}| }{|\Hom(S^{(\alpha)},S^{(n-c)})|}\frac{|\Ext^1(S^{(\alpha)},S^{(m-a)})_{S^{(\rho)}}|}{|\Hom^1(S^{(\alpha)},S^{(m-a)})|} \\
			&\quad\times\frac{1}{|\Aut(S^{(\alpha)})|}[S^{(\rho)}\oplus S^{(c)}[1]].
		\end{align*}
		Note that the following holds for any $r\ge1$ and $\alpha\ne\emptyset$ by \cite[Corollary 5.2]{LRW25},
		\begin{align*}
			\sum_{\la}\varphi_{\ell(\la)-1}(q)|\Ext^1(S^{(\alpha)},S^{(r)})_{S^{(\la)}}|=0.
		\end{align*}
		So the terms vanish in $(\sqq-\sqq^{-1})\widetilde{P}_a*\widetilde{\Theta}_{m-a,n}$ unless $n=c$ or $\alpha=\emptyset$. If $n=c$, then $\la=\alpha$; if $\alpha=\emptyset$, then $\la=(n-c),n-c=a$ and $\rho=(m-a)$. Hence, we obtain
		\begin{align*}
			(\sqq-\sqq^{-1})\widetilde{P}_a*\widetilde{\Theta}_{m-a,n}=&[S^{(m-a)}\oplus S^{(n-a)}[1]]+
			\sum_{\la\vdash a} {\varphi_{\ell(\la)-1}(q)}\sum_{\rho} \frac{|\Ext^1(S^{(\la)},S^{(m-a)})_{S^{(\rho)}}|}{|\Hom^1(S^{(\la)},S^{(m-a)})|}
			\\
			&\qquad\times\frac{1}{|\Aut(S^{(\la)})|}[S^{(\rho)}\oplus S^{(n)}[1]].
		\end{align*}
		Note that
		\begin{align*}
			(\sqq-\sqq^{-1})\sum_{\substack{1\le a\le \min(m,n)\\(a,a)\ne(m,n)}} \sqq^a[a]_\sqq\widetilde{\Theta}_{m-a,n-a}-\sum_{1\le a\le m}\sqq^a [a]_\sqq [S^{(m-a)}\oplus S^{(n-a)}[1]]=-\delta_{m,n}\sqq^m{[m]_\sqq}.
		\end{align*}
		It remains to prove
		\begin{align}
			\label{eq:Theta-TP}    &m\widetilde{\Theta}_{m,n}-\sum_{{1\le a\le m}} \sqq^a[a]_\sqq \sum_{\la\vdash a} {\varphi_{\ell(\la)-1}(q)}
			\\\notag&\qquad\cdot\sum_{\rho} \frac{|\Ext^1(S^{(\la)},S^{(m-a)})_{S^{(\rho)}}|}{|\Hom^1(S^{(\la)},S^{(m-a)})|} \frac{1}{|\Aut(S^{(\la)})|}[S^{(\rho)}\oplus S^{(n)}[1]]=0.
		\end{align}
		We shall prove \eqref{eq:Theta-TP} by considering the coefficients of  $[S^{(\rho)}\oplus S^{(n)}[1]]$.
		
		For the case $\ell(\rho)=1$, we have $\rho=(m)$ and $\la=(a)$. Then the coefficient of $[S^{(\rho)}\oplus S^{(n)}[1]]$ is equal to
		\begin{align*}
			&\frac{m}{\sqq-\sqq^{-1}}-\sum_{1\le a\le m} \sqq^a [a]_\sqq\frac{|\Aut(S^{(m-a)})|}{|\Aut(S^{(m)})|}
			\\=&\frac{m}{\sqq-\sqq^{-1}}-\sqq^m[m]_\sqq\frac{1}{(q-1)q^{m-1}}-\sum_{1\le a\le m-1} \frac{q^a-1}{\sqq-\sqq^{-1}}\frac{q^{m-a}(1-q^{-1})}{q^{m}(1-q^{-1})}
			\\=&\frac{m}{\sqq-\sqq^{-1}}-\frac{q^m-1}{(\sqq-\sqq^{-1})(q-1)q^{m-1}}-\sum_{1\le a\le m-1} \frac{1-q^{-a}}{\sqq-\sqq^{-1}}
			\\=&\frac{m}{\sqq-\sqq^{-1}}-\frac{q-q^{1-m}}{(\sqq-\sqq^{-1})(q-1)}-\frac{m-1}{\sqq-\sqq^{-1}}+ \frac{1-q^{1-m}}{(q-1)(\sqq-\sqq^{-1})}
			=0.
		\end{align*}
		
		For the case $\ell(\rho)>1$, the coefficient of $[S^{(\rho)}\oplus S^{(n)}[1]]$ is also equal to 0 by Case $(b)$ in \cite[Proposition 6.7]{LRW25}.
	\end{proof}

	Similarly, we can write $\widetilde{E}(y,z)$, $\widetilde{H}(y,z)$ in terms of ${P}_1(y)$, ${P}_2(z)$; see the following proposition.
	
	\begin{proposition}
		\label{prop:EHP}
		We have
		\begin{align}
			\label{eq:EP}
			\widetilde{E}(y,z)=&\exp\big(\sum_{r\ge1}\frac{(-1)^{r-1}}{r}\widetilde{P}_ry^r\big)  \exp\big(\sum_{t\ge1}\frac{(-1)^{t-1}}{t}\widetilde{P}'_tz^t\big)\exp_q\big( \frac{yz}{1-q}\big),
			\\
			\label{eq:HP}
			\widetilde{H}(y,z)=&\exp\big(\sum_{r\ge1}\frac{1}{r}\widetilde{P}_ry^r\big)  \exp\big(\sum_{t\ge1}\frac{1}{t}\widetilde{P}'_tz^t\big)\exp_q\big( \frac{yz}{1-q}\big).
		\end{align}
	\end{proposition}
	
	\begin{proof}
		Using Proposition \ref{e-h}, we only prove \eqref{eq:EP}.
		
		By Proposition \ref{theta-e-h}, we have
		\begin{align*}
			\frac{\widetilde{E}(-y,-z)}{\widetilde{E}(-qy,-qz)}=\frac{ \widetilde{\Theta}(y,z)}{(1-qyz)^2}.
		\end{align*}
		Replacing $y,z$ by $q^my,q^mz$ respectively, we obtain
		\begin{align}\label{e-prod}
			\widetilde{E}(-y,-z)=\prod_{m\ge 0}\frac{\widetilde{E}(-q^my,-q^mz)}{\widetilde{E}(-q^{m+1}y,-q^{m+1}z)}=\frac{\prod_{m\ge0}\widetilde{\Theta}(q^my,q^mz)}{\prod_{m\ge0}(1-q^{2m+1}yz)^2}.
		\end{align}
		We deduce from Proposition \ref{T-P} that
		\begin{align*}
			\widetilde{\Theta}(y,z)=\frac{\exp\big(\sum_{r\ge1}\frac{q^r}{r}\widetilde{P}_ry^r\big) }{\exp\big(\sum_{r\ge1}\frac{1}{r}\widetilde{P}_ry^r\big) } \cdot \frac{\exp\big(\sum_{t\ge1}\frac{q^t}{t}\widetilde{P}'_tz^t\big)}{\exp\big(\sum_{t\ge1}\frac{1}{t}\widetilde{P}'_tz^t\big)}\cdot \exp\big( \sum_{k\ge1}\frac{1-q^k}{k}(yz)^k\big),
		\end{align*}
		and it gives that
		\begin{align*}
			\prod_{m\ge0}\widetilde{\Theta}(q^my,q^mz)=&\exp\big(\sum_{r\ge1}-\frac{1}{r}{\widetilde{P}}_ry^r\big)\exp\big(\sum_{t\ge1}\-\frac{1}{t}\widetilde{P}'_tz^t\big)\exp\big( \sum_{m\ge0}\sum_{k\ge1}\frac{1-q^k}{k}(q^{2m}yz)^k\big)
			\\=&\exp\big(\sum_{r\ge1}-\frac{1}{r}{\widetilde{P}}_ry^r\big)\exp\big(\sum_{t\ge1}\-\frac{1}{t}\widetilde{P}'_tz^t\big)\exp\big( \sum_{m\ge0}\sum_{k\ge1}\frac{q^{2mk}-q^{(2m+1)k}}{k}(yz)^k\big)
			\\=&\exp\big(\sum_{r\ge1}-\frac{1}{r}{\widetilde{P}}_ry^r\big)\exp\big(\sum_{t\ge1}\-\frac{1}{t}\widetilde{P}'_tz^t\big)\exp\big( \sum_{m\ge0}\ln\frac{1-q^{2m+1}yz}{1-q^{2m}yz}
			\big)
			\\=&\exp\big(\sum_{r\ge1}-\frac{1}{r}{\widetilde{P}}_ry^r\big)\exp\big(\sum_{t\ge1}\-\frac{1}{t}\widetilde{P}'_tz^t\big)\cdot\big( \prod_{m\ge0}\frac{1-q^{2m+1}yz}{1-q^{2m}yz}
			\big).
		\end{align*}
		
		Plugging it into \eqref{e-prod}, we have
		\begin{align*}
			\widetilde{E}(-y,-z)=&\exp\big(\sum_{r\ge1}-\frac{1}{r}{\widetilde{P}}_ry^r\big)\exp\big(\sum_{t\ge1}\-\frac{1}{t}\widetilde{P}'_tz^t\big)\frac{ \prod_{m\ge0}\frac{1-q^{2m+1}yz}{1-q^{2m}yz}
			}{\prod_{m\ge0}(1-q^{2m+1}yz)^2}
			\\=&\exp\big(\sum_{r\ge1}-\frac{1}{r}{\widetilde{P}}_ry^r\big)\exp\big(\sum_{t\ge1}\-\frac{1}{t}\widetilde{P}'_tz^t\big)\prod_{m\ge0}\frac{ 1}{1-q^{m}yz}
			\\=&\exp\big(\sum_{r\ge1}-\frac{1}{r}{\widetilde{P}}_ry^r\big)\exp\big(\sum_{t\ge1}\-\frac{1}{t}\widetilde{P}'_tz^t\big)\exp_q\big(\frac{yz}{1-q}\big).
		\end{align*}
		The proof is completed.
	\end{proof}
	
	\begin{corollary}
		\label{lem:PE}
		We have
		\begin{align*}
			\widetilde{P}(-y) \widetilde{E}(y,z)
			&=  \frac{\partial}{\partial y}\widetilde{E}(y,z)  - \widetilde{E}(y,z) \sum_{k\ge 1} \frac {y^{k-1} z^{k}}{ 1-q^k},
		\end{align*}
		\begin{align*}
			\widetilde{P}(y) \widetilde{H}(y,z)
			&=  \frac{\partial}{\partial y}\widetilde{H}(y,z)  - \widetilde{H}(y,z) \sum_{k\ge 1} \frac {y^{k-1} z^{k}}{ 1-q^k}.
		\end{align*}
		
		\begin{align*}
			\widetilde{P}'(-z) \widetilde{E}(y,z)
			&=  \frac{\partial}{\partial z}\widetilde{E}(y,z)  - \widetilde{E}(y,z) \sum_{k\ge 1} \frac {y^{k} z^{k-1}}{ 1-q^k},
		\end{align*}
		\begin{align*}
			\widetilde{P}'(z) \widetilde{H}(y,z)
			&=  \frac{\partial}{\partial z}\widetilde{H}(y,z)  - \widetilde{H}(y,z) \sum_{k\ge 1} \frac {y^{k} z^{k-1}}{ 1-q^k}.
		\end{align*}
		
	\end{corollary}
	
	\begin{proof}
		The first formula follows by first replacing $\exp_q$ in the first formula in Proposition~\ref{prop:EHP} via the identity \eqref{exp=exp}, and then applying the differentiation $\frac{\partial}{\partial y}$. The proofs of other formulas are entirely similar.
	\end{proof}

	Now, we can complete the proof of Theorem \ref{prop:transition}.
	
	\begin{proof}[Proof of Theorem \ref{prop:transition}]
		
		We only prove \eqref{eq:Theta-Theta12} since the other two are similar.
		
		By transforming \cite[Chapter 3, (2.10)]{Mac95} to the classic Hall algebra $\widetilde{\ch}(\bfk \QJ)$,   
		we know $$\Theta_1(y)=\exp\big(\sum_{r\ge1}\frac{q^r-1}{r}\widetilde{P}_ry^r\big),\quad \Theta_2(z)=\exp\big(\sum_{t\ge1}\frac{q^t-1}{t}\widetilde{P}'_tz^t\big). $$
		Then \eqref{eq:Theta-Theta12} follows from \eqref{theta-p}.
	\end{proof}

	%



\end{document}